\documentclass[%
 letterpaper,
  twocolumn,
  colorlinks,
]{preprint}

\usepackage[utf8]{inputenc}	
\usepackage[T1]{fontenc}
\usepackage{dsfont}

\usepackage{amsmath}
\usepackage{amsfonts}
\usepackage{amssymb}
\usepackage{amscd}
\usepackage{amsthm}
\usepackage{bbm}

\usepackage[english]{babel}

\usepackage{graphicx} 
\usepackage{caption}
\usepackage{subcaption}

\usepackage{tikz}
\usepackage{pgfplots}
\pgfplotsset{compat = 1.17}
\pgfkeys{/pgf/number format/.cd,1000 sep={\,}}

\usepackage{algorithm}
\usepackage{etoolbox}
\let\classAND\AND
\let\AND\relax
\usepackage{algorithmic}

\let\AND\classAND
\AtBeginEnvironment{algorithmic}{\let\AND\algoAND}
\floatname{algorithm}{Algorithm}

\usepackage{xspace}
\usepackage{hyperref}    
\usepackage{nicefrac} 
\usepackage{cancel}
\usepackage{empheq} 
\usepackage{cancel}

\DeclareMathOperator*{\argmin}{arg\,min}
\newcommand{\trans}{\ensuremath{\mkern-1.5mu\mathsf{T}}}
\newcommand{\frob}{\ensuremath{\mathsf{F}}}
\newcommand{\herm}{\ensuremath{\mathsf{H}}}
\DeclareMathOperator{\vecm}{vec}
\DeclareMathOperator{\mspan}{span}
\DeclareMathOperator{\trace}{tr}

\def\Sys{\mathcal{S}}

\def\Sysonlyc{\Sys_{1}}
\def\Sysonlycred{\Sys_{1,r}}

\def\Sysred{\Sys_r}

\def\wh#1{\widehat{#1}}

\mathcode`\@="8000
{\catcode`\@=\active \gdef@{\mkern1mu}}

\def\CH{\mathcal{H}}

\def\CL{\mathcal{L}}
\def\CJ{\mathcal{J}}

\def\R{{\mathbbm R}}
\def\C{{\mathbbm C}}

\def\Crr{\C^{r\times r}}

\def\Rn{\R^n}
\def\Rm{\R^m}
\def\Rp{\R^p}
\def\Rr{\R^r}

\def\Rnn{\R^{n\times n}}
\def\Rnm{\R^{n\times m}}
\def\Rpn{\R^{p\times n}}
\def\Rpm{\R^{p\times m}}

\def\Rrr{\R^{r\times r}}
\def\Rpr{\R^{p\times r}}
\def\Rrm{\R^{r\times m}}
\def\Rnr{\R^{n\times r}}

\DeclareOldFontCommand{\bf}{\normalfont\bfseries}{\mathbf}
\def\Bzero{{\bf 0}}
\def\BA{{\bf A}}  \def\BAs{{\bf A}\kern-1pt}
\def\BB{{\bf B}} 
\def\BC{{\bf C}} \def\Bc{{\bf c}}
\def\BD{{\bf D}} 
 
\def\BF{{\bf F}} 
\def\BG{{\bf G}} 
\def\BH{{\bf H}} \def\Bh{{\bf h}}
\def\BI{{\bf I}}

\def\BM{{\bf M}}

\def\BP{{\bf P}}  \def\BPs{\BP\kern-1pt} 
\def\BQ{{\bf Q}} 
 
\def\BS{{\bf S}} 
\def\BT{{\bf T}} 
\def\BU{{\bf U}} \def\Bu{{\bf u}}
\def\BV{{\bf V}} \def\Bv{{\bf v}}
\def\BW{{\bf W}} 
\def\BX{{\bf X}} \def\Bx{{\bf x}}
\def\BY{{\bf Y}} \def\By{{\bf y}}
\def\BZ{{\bf Z}}

\def\BQtwok{\BQ_{2}^{(k)}}
\def\BZtwok{\BZ_{2}^{(k)}}

\def\BPr{{\bf P}_r} 
\def\BQr{{\bf Q}_r} 
\def\BQoner{{\bf Q}_{1,r}} 
\def\BAr{{\bf A}_r} 
\def\BBr{{\bf B}_r} 
\def\BCr{{\bf C}_r}

\def\BMr{{\bf M}_r} 
\def\Bxr{{\bf x}_r} 
\def\BVr{{\bf V}_r} \def\BWr{{\bf W}_r}
\def\yr{{y}_r}
\def\Byr{{\By}_r}
\def\Byone{{\By}_1}
\def\Byoner{{\By}_{1,r}}

\def\BMoner{{\bf M}_{1,r}} 
 
\def\BMkr{{\bf M}_{k,r}} 
\def\BMpr{{\bf M}_{p,r}} 

\def\BAe{{\bf A}_e} 
\def\BPe{{\bf P}_e} 
\def\BQe{{\bf Q}_e}
\def\BQonee{{\bf Q}_{1,e}} 
\def\BMe{{\bf M}_e}
\def\BMonee{{\bf M}_{1,e}} 
 
\def\BMke{{\bf M}_{k,e}} 
\def\BMpe{{\bf M}_{p,e}} 
\def\BBe{{\bf B}_e}  
\def\BCe{{\bf C}_e}  

\def\Bxe{{\bf x}_e}

\def\gradA{\nabla_{\BAr}\CJ}

\def\gradB{\nabla_{\BBr}\CJ}
\def\gradC{\nabla_{\BCr}\CJ}

\def\gradMk{\nabla_{\BMkr}\CJ}

\def\pertSys{\Delta_{\Sysred}}
\def\pertA{\Delta_{\BAr}}

\def\pertB{\Delta_{\BBr}}
\def\pertC{\Delta_{\BCr}}

\def\pertMk{\Delta_{\BMkr}}
\def\pertZ{\Delta_{\BZ}}
\def\pertX{\Delta_{\BX}}
\def\pertQ{\Delta_{\BQr}}
\def\pertP{\Delta_{\BPr}}

\newtheorem{theorem}{Theorem}[section]

\newtheorem{lemma}{Lemma}[section]

\newtheorem{remark}{Remark}[section]
\newtheorem{definition}{Definition}[section]

\newcommand{\MOR}{\ensuremath{\mathsf{MOR}\xspace}}
\newcommand{\ROM}{\ensuremath{\mathsf{ROM}\xspace}}
\newcommand{\FOM}{\ensuremath{\mathsf{FOM}\xspace}}
\newcommand{\LTI}{\ensuremath{\mathsf{LTI}\xspace}}

\newcommand{\LQO}{\ensuremath{\mathsf{LQO}\xspace}}
\newcommand{\QO}{\ensuremath{\mathsf{QO}\xspace}}
\newcommand{\IRKA}{\ensuremath{\mathsf{IRKA}\xspace}}

\newcommand{\FONC}{\ensuremath{\mathsf{FONC}\xspace}}
\newcommand{\TSIA}{\ensuremath{\mathsf{TSIA}\xspace}}

\newcommand{\MIMO}{\ensuremath{\mathsf{MIMO}\xspace}}

\newcommand{\BTr}{\ensuremath{\mathsf{BT}\xspace}}
\newcommand{\LQOTSIAdiag}{\ensuremath{\mathsf{LQO}\mbox{-}\mathsf{TSIA_{diag}}}}
\newcommand{\LQOTSIAtrunc}{\ensuremath{\mathsf{LQO}\mbox{-}\mathsf{TSIA_{trunc}}}}
\newcommand{\LQOTSIAeigs}{\ensuremath{\mathsf{LQO}\mbox{-}\mathsf{TSIA_{eigs}}}}
\newcommand{\LQOBT}{\ensuremath{\mathsf{LQO}\mbox{-}\mathsf{BT}}}
\newcommand{\relerr}{\ensuremath{\operatorname{relerr}}}

\definecolor{matlabblue}{HTML}{0072BD}
\definecolor{matlaborange}{HTML}{D95319}
\definecolor{matlabyellow}{HTML}{EDB120}
\definecolor{matlabpurple}{HTML}{7E2F8E}
\definecolor{matlabgreen}{HTML}{77AC30}
\definecolor{matlablightblue}{HTML}{4DBEEE}
\definecolor{matlabred}{HTML}{A2142F}

\tikzstyle{sline} = [
  solid,
  line width = 1.5pt
]

\newcommand{\plotfontsize}{\footnotesize}

\pgfplotscreateplotcyclelist{plotlist}{
  {black, sline},
  {matlabblue, sline, dashed},
  {matlaborange, sline, loosely dotted},
  {matlabpurple, sline, dashed},
  {matlabgreen, sline, dotted},
  }

\pgfplotscreateplotcyclelist{errorplotlist}{
  {black, sline},
  {matlabblue, sline, mark=o},
  {matlaborange, sline, mark=+},
  {matlabpurple, sline, mark=x},
  {matlabgreen, sline, mark=star},
}

\pgfplotscreateplotcyclelist{convplotlist}{
  {black, sline},
  {matlabblue, sline, mark=o},
  {matlabblue, sline, mark=+},
  {matlaborange, sline, mark=o},
  {matlaborange, sline, mark=+},
  {matlabpurple, sline, mark=o},
  {matlabpurple, sline, mark=+},
}

\begin{document}
  

\title{$\CH_2$ optimal model reduction of linear systems with multiple quadratic outputs}

\author[$\dagger$]{Sean Reiter}
\affil[$\dagger$]{%
  Courant Institute of Mathematical Sciences, New York University,
  New York, NY 10012, USA.\authorcr
  \email{s.reiter@nyu.edu}, \orcid{0000-0002-7510-1530}
}

\author[$\star$]{Igor Pontes Duff}
\affil[$\star$]{%
  Max Planck Institute for Dynamics of Complex Technical Systems,
  Sandtorstr. 1, 39106 Magdeburg, Germany.\authorcr
  \email{pontes@mpi-magdeburg.mpg.de}, \orcid{0000-0001-6433-6142}
}
  
\author[$\ast$]{Ion Victor Gosea}
\affil[$\ast$]{
  Max Planck Institute for Dynamics of Complex Technical Systems,
  Sandtorstr. 1, 39106 Magdeburg, Germany.\authorcr
  \email{gosea@mpi-magdeburg.mpg.de}, \orcid{0000-0003-3580-4116}
}

\author[$\ddagger$]{Serkan Gugercin}
\affil[$\ddagger$]{%
  Department of Mathematics and Division of Computational Modeling and Data
  Analytics, Academy of Data Science, Virginia Tech,
  Blacksburg, VA 24061, USA.\authorcr
  \email{gugercin@vt.edu}, \orcid{0000-0003-4564-5999}
}

\shorttitle{$\CH_2$ optimal model reduction of linear quadratic output systems}
\shortauthor{S. Reiter, I. Pontes~Duff, I.~V. Gosea, S. Gugercin}
\shortdate{\today}
  
\keywords{%
linear quadratic output systems, $\CH_2$ optimal, model-order reduction, Gramians, two-sided iteration, optimality conditions, gradients
}

\msc{%
    15A24, 
    46N10, 
    49K15, 
    93A15, 
    93C10, 
    93C80  
}

\abstract{%
In this work, we consider the $\CH_2$-optimal model reduction of dynamical systems that are linear in the state equation with up to a quadratic nonlinearity in the output equation. As our primary theoretical contributions, we derive gradients of the squared $\CH_2$ system error with respect to the reduced model quantities and, from the stationary points of these gradients, introduce Gramian-based first-order necessary conditions for the $\CH_2$-optimal approximation of a linear quadratic output (\LQO{}) system. The resulting $\CH_2$-optimality framework generalizes the analogous Gramian-based optimality framework for purely linear systems. Computationally, we demonstrate how to enforce the necessary optimality conditions using the Petrov-Galerkin projection; the corresponding projection matrices are derived from a pair of Sylvester equations. Based on this result, we propose an iteratively corrected algorithm for the $\CH_2$-optimal model reduction of \LQO{} systems, which we refer to as the linear quadratic output two-sided iteration algorithm (\LQO{}-\TSIA{}). Numerical examples are included to illustrate the effectiveness of the proposed computational method against other existing approaches.
}

\novelty{%
}

\maketitle


\section{Introduction}
\label{sec:intro}
The focus of this work is a class of weakly nonlinear dynamical systems: those that are linear time-invariant in the state equation with up to quadratic terms in the output equation. We refer to these systems as \emph{linear quadratic-output} $\left(\LQO{}\right)$ systems.
Dynamical systems with quadratic-output functions appear naturally in applications where one is interested in observing or simulating response quantities computed as the product of time- or frequency-components of the state. 
Relevant examples of quadratic outputs include quantities pertaining to power or energy~\cite{HolNSU25, Pul23}, quadratic cost functionals in optimal control or design problems~\cite{DiazHGA23, YueM12, YueM13}, and observables that capture the variance or deviation of the state coordinates from a point of reference~\cite{VanVNLM12, AumW23, ReiW24, PulA19}.
Our interest in \LQO{} systems is primarily motivated by these applications where the dynamics of the problem are linear, but the quantities of interest are quadratic functions of the state.
In the interest of space, we refer the reader to~\cite[Sec.~2]{YueM12},~\cite[Sec.~2]{ReiW24}, or~\cite[Sec.~5.2]{Rei25} for a more comprehensive discussion of specific examples of systems with quadratic outputs.

The accurate modeling of complex physical phenomena often requires dynamical systems having a very large state-space dimension, e.g., on the order of $10^6$ or greater. 
In this large-scale setting, direct calculations involving the full-order model $(\FOM)$ are prohibitively costly, and system approximation becomes desirable.
\emph{Model order reduction} $(\MOR)$ is the procedure by which one replaces a large-scale dynamical system with a comparatively low-order surrogate model. The resultant \emph{reduced-order model} $(\ROM)$ should capture the significant input-to-output response characteristics of the original system in an appropriate sense, so that it can be used as a computationally efficient and faithful proxy for the \FOM{} in applications of, e.g., prediction, simulation, or control. 

In the \emph{purely} linear setting---that is, both the internal dynamics \emph{and} output equation depend only linearly on the state coordinates---there are a variety of \MOR{} techniques at one's disposal; we refer the reader to~\cite{Ant05, BenMS05, BenOCW17, AntBG20} and the collection of references therein for an overview of approximation methods for linear dynamical systems. 
In recent years, there has been an increased amount of study dedicated towards the \MOR{} of \LQO{} systems; see, e.g.,~\cite{VanM10, VanVNLM12, PulA19, GosA19, GosG20, BenGPD21, Pul23, DiazHGA23, PrzDGB24, SonZXUS24, ReiW24}.
Some of these works tackle the \LQO-\MOR{} problem by rewriting the governing equations of a multi-input, single-output \LQO{} system as a \emph{wholly linear} multi-input, multi-output system~\cite{VanM10, VanVNLM12}, or as a quadratic-bilinear system with a \emph{single} linear output~\cite{PulA19}.
By contrast, the works~\cite{BenGPD21, GosA19, GosG20, DiazHGA23, Pul23, SonZXUS24, PrzDGB24} devise approaches that leverage the quadratic-output structure \emph{directly}---that is, without any intermediate lifting or linearization---to determine suitable approximation subspaces for order reduction.
We highlight the work of Benner et al.~\cite{BenGPD21}, which proposes a novel algebraic observability Gramian, that we refer to as the \emph{quadratic-output observability Gramian}, and related $\CH_2$ system norm based on the Volterra kernels of an \LQO{} system. 
This theory leads to a balanced truncation model reduction algorithm based on the linear reachability and quadratic-output observability Gramians; extensions of this methodology to handle systems of differential-algebraic equations and frequency or time-limited model reduction were proposed in~\cite{PrzDGB24} and~\cite{SonZXUS24}, respectively. 
The works~\cite{GosA19, GosG20, DiazHGA23, ReiW24} consider model reduction based on the rational interpolation of transfer functions in the frequency domain.

It is proven in~\cite[Theorem~3.4]{BenGPD21} that the $\CH_2$ linear quadratic-output system error provides an \emph{a posteriori} upper bound on the $\CL_\infty$ output error; this bound holds for \emph{any} \LQO-\ROM{}, regardless of the computational strategy used to produce it.  
This is a natural motivator for using the $\CH_2$ model error as a design objective in \LQO-\MOR, particularly if one is interested in ensuring that the `worst case' error in the approximate output is uniformly small over all inputs.
Outside of~\cite{GosA19} and~\cite{BenGPD21}, which consider the $\CH_2$ norm as a performance measure but not explicitly as a design objective, the $\CH_2$-optimal model reduction of \LQO{} systems has not been considered in the current literature.
Specifically,~\cite{GosA19} proposes an algorithm heuristically motivated by \emph{linear} $\CH_2$-optimal model reduction, but without any explicit connection to the $\CH_2$-optimal model reduction of \LQO{} systems.

Motivated by the bound in~\cite[Theorem~3.4]{BenGPD21}, the focus of this work is on the structure-preserving and $\CH_2$-optimal model reduction of linear systems with quadratic-output functions.
By structure-preserving, we mean the $\CH_2$-optimal model reduction of an \LQO{} system by another, lower-order \LQO{} reduced model.
To our knowledge, this is the first time the $\CH_2$-optimal \LQO{} model reduction problem has been investigated in the literature.
For strictly linear systems, $\CH_2$-optimal model reduction is well studied; see, e.g.,~\cite{GugAB08, XuZ11, VanGPA08, MeiL67, Wil70, AntBG20, BauBF14}, and the references therein. 
The $\CH_2$-optimal model reduction problem has also been considered for other classes of weakly nonlinear dynamical systems, such as bilinear systems; see, e.g.,~\cite{ZhaL02, BenB12, FlaG15, BenGG18}.
Even in the linear setting, finding a global minimizer of the $\CH_2$ error is a non-convex optimization problem. 
As a consequence, the most common practice in the $\CH_2$ landscape typically relies on the identification of \ROM{}s that satisfy \emph{first-order necessary conditions} $(\FONC{}\mbox{s})$ for \emph{local} $\CH_2$ optimality. 
The two most well-known optimality frameworks for $\CH_2$-optimal model reduction of fully linear systems are derived from the interpolation-based \FONC{}s of Meier and Luenberger~\cite{MeiL67, GugAB08} and the Gramian-based \FONC{}s of Wilson~\cite{Wil70, VanGPA08}.
The major theoretical contribution that we present in this work is the establishment of \emph{Gramian-based} $\CH_2$-optimality conditions for the model reduction of \LQO{} systems.
Namely, our significant contributions are:
\begin{enumerate}
    \item \Cref{thm:gradients}, which derives gradients of the squared $\CH_2$ system error with respect to the state-space matrices of the \LQO-\ROM{} as parameters. 
    The stationary points of these gradients directly yield Gramian-based \FONC{}s for $\CH_2$ optimality, which are presented in~\Cref{thm:foncs_gramians}.
    These results constitute a set of structured $\CH_2$-optimality conditions for the \LQO-\MOR{} problem, and 
    generalize the analogous Gramian-based $\CH_2$-opti\-mality conditions in linear model reduction~\cite{Wil70, VanGPA08} to the \LQO{} setting.
    \item Also in~\Cref{thm:foncs_gramians}, we show that a $\CH_2$-optimal \LQO-\ROM{} is necessarily defined by Petrov-Galerkin projection. The relevant projection matrices are obtainable as solutions to a pair of Sylvester equations.
    \item Based on this theoretical optimality framework, we propose an iteratively corrected algorithm for $\CH_2$-opt\-imal \LQO-\MOR{} in~\Cref{alg:lqo_tsia}; we call this the linear quadratic-output two-sided iteration algorithm $(\LQO\-\mbox{-}\TSIA)$.
    At each iteration, the method obtains projection matrices from the aforementioned pair of Sylvester equations corresponding to the previous model iterate. 
    If the algorithm converges, the necessary $\CH_2$-optimality conditions of~\Cref{thm:foncs_gramians} will be satisfied by the outputted \LQO-\ROM{}.
\end{enumerate}

The particular organization of the manuscript is as follows: In~\Cref{sec:background}, we review the necessary systems theory background of \LQO{} systems and introduce several different characterizations of the $\CH_2$ system norm. These are central to the results developed in~\Cref{sec:h2_opt}, where we present the main theoretical contributions of this work. These are the derivation of gradients of the squared $\CH_2$ system error and resulting Gramian-based \FONC{}s for $\CH_2$ optimality in \LQO-\MOR.
\Cref{sec:algorithm} presents \LQO-\TSIA, an iterative computational algorithm for \LQO-\MOR{} based on the previously established $\CH_2$-optimality framework. 
To validate the proposed approach, in~\Cref{sec:numerics} we test our method on an example where a quadratic output naturally occurs from the discretization of a quadratic cost function.
\Cref{sec:conclusions} concludes the work and looks toward future research endeavors.

\section{Background and preliminaries}
\label{sec:background}
\subsection{Linear systems with quadratic outputs}
\label{ss:lqo_bg}

Throughout this work, we consider multi-input, multi-output $(\MIMO)$ dynamical systems that are linear in the state equation with a quadratic-output $(\QO)$ term. In its state-space formulation, such a system is described by the following 
equations
\begin{align}\label{eq:lqosys}
\Sys: \begin{cases}  \dot\Bx(t)=\BA\Bx(t)+\BB \Bu(t),\quad\quad \Bx(0)=\Bzero,\\[6pt]
\hspace{0.5mm} \By(t)=\underbrace{\BC\Bx(t)}_{\coloneqq\By_1(t)} + \underbrace{\begin{bmatrix}
    \Bx(t)^{\trans} \BM_1 \Bx(t)\\
    \vdots\\
    \Bx(t)^{\trans} \BM_p \Bx(t)
\end{bmatrix}}_{\coloneqq\By_2(t)},
\end{cases}
\end{align}
where $\BA,$ $\BM_1,\ldots,\BM_p \in \Rnn$, $\BB\in\Rnm$, and $\BC \in \Rpn$. In~\eqref{eq:lqosys} $\Bx\colon[0,\infty)\to\Rn$ contains the state coordinates; $\Bu\colon[0,\infty)\to\Rm$ and $\By\colon[0,\infty)\to\Rp$ are the inputs and outputs of the system, respectively. 
We assume henceforth that $\Sys$ is \emph{asymptotically stable}, i.e., the eigenvalues of $\BA$, denoted $\lambda(\BA)$, lie in the open left half of the complex plane, denoted $\C_-$.
Because $\BM_k$ can always be replaced by its symmetric part, we assume without loss of generality that $\BM_k=\BM_k^{\trans}$ for all $k=1,\ldots, p$.
We refer to a system of the form~\eqref{eq:lqosys} as an \LQO{} system, and represent such a system using the notation
$\left(\BA, \BB, \BC, \BM_1,\ldots,\BM_p\right).$
The \QO{} term $\By_2$ in~\eqref{eq:lqosys} can be re-written via a Kronecker product of the state
\begin{align}
\label{eq:kronqo}
    \By_2(t) =\BM \left(\Bx(t)\otimes \Bx(t)\right),\ \ \BM=\begin{bmatrix}\mbox{vec}\left(\BM_1\right)^{\trans} \\\vdots \\ \mbox{vec}\left(\BM_p\right)^{\trans}\end{bmatrix}\in\R^{p\times n^2},
\end{align}
where $\mbox{vec}\colon\Rnn\to\R^{n^2}$ denotes the \emph{vectorization} operator. For compactness and ease of theoretical development, we switch freely between these representations as necessary.
Depending on the application of interest, in~\eqref{eq:lqosys} we allow for $\BC=\Bzero_{p\times n}$, the $p \times n$ zero matrix, in which case the output equation is \emph{purely} quadratic.
On the other hand,
if $\BM_k=\Bzero_{n\times n}$ for all $k=1,\ldots,p$, the system~\eqref{eq:lqosys} reduces to a standard linear time-invariant (\LTI) system. We denote the system~\eqref{eq:lqosys} by $\Sysonlyc$ in this instance.

In practical applications, the state dimension $n$ of the system in~\eqref{eq:lqosys} may be large enough so that any repeated computation involving the \FOM{} is prohibitively expensive. 
\MOR{} seeks to replace the original large-scale model with a lower-dimensional \ROM{} that can be used as a cheap-to-evaluate surrogate in numerical simulation, devising control laws, solving optimization tasks, etc.
In this work, we consider the construction of reduced models that retain the \LQO{} structure, i.e.,
\begin{align}\label{eq:lqosys_red}
    \Sysred: \begin{cases}  \dot{\Bx}_r(t)=\BAr\Bxr(t)+\BBr \Bu(t),\\[6pt]
    \hspace{0.5mm} \By_r(t)=\BCr\Bxr(t) + \begin{bmatrix}
    \Bx(t)^{\trans} \BM_{1,r} \Bx(t)\\
    \vdots\\
    \Bx(t)^{\trans} \BM_{p,r} \Bx(t)
\end{bmatrix} ,
    \end{cases}
\end{align}
where the matrices $\BAr,$ $\BMoner,\ldots,\BMpr\in\Rrr$, $\BBr\in\Rrm,$ and $\BCr\in\Rpr$ are of a significantly reduced dimension $1\leq r \ll n$,
$\Bxr\colon[0,\infty)\to\Rr$ is the reduced state vector, and $\Byr\colon[0,\infty)\to\Rp$ is the approximate output.
We would like the surrogate model~\eqref{eq:lqosys_red} to be a good approximation in the sense that it accurately recreates the input-to-output response of the original full-order system in~\eqref{eq:lqosys}. 
In other words, the reduced output $\Byr$ should be a faithful replication of $\By$ in the sense that $\|\By-\Byr\|$ is small in an appropriate norm $\|\cdot\|$ over a range of admissible inputs $\Bu$.

The general framework we consider is that of model reduction using \emph{projection}.
Consider an \LQO{} system as in~\eqref{eq:lqosys}. In projection-based model reduction, the problem resolves to choosing left and right reduction bases $\BWr\in\Rnr$ and $\BVr\in \Rnr$ 
so that $\Bx\approx \BVr\Bxr$, the matrix $\BWr^{\trans}\BVr$ is nonsingular, and the Petrov-Galerkin condition
\begin{equation*}
    {\BW_r}^{\trans}\left({\BV_r}\dot{\Bx}_r(t) - \BA{\BV_r} \Bxr(t) - \BB @ \Bu(t)\right)=\Bzero
\end{equation*}
is satisfied.
The resulting order-$r$ \LQO{} reduced model of the form~\eqref{eq:lqosys_red} is determined by the reduced order matrices
\begin{equation}
\begin{alignedat}{2}
\label{eq:pg_proj}
    \BAr&= \left(\BWr^{\trans}\BVr\right)^{-1}\BW_r^{\trans}\BA
    {\BV_r} &&\in \R^{r\times r},\\
    \BBr &= \left(\BWr^{\trans}\BVr\right)^{-1}\BW_r^{\trans}\BB &&\in \R^{r\times m},\\
    \BCr &= \BC{\BV_r}&&\in \R^{p\times r},\\
    \mbox{and}~~\BMkr &=
    \BV_r^{\trans}\BM_k {\BV_r} &&\in \R^{r\times r}.
\end{alignedat}
\end{equation}
The approximation quality of the reduced model hinges upon the underlying projection subspaces $\mspan(\BWr)$ and $\mspan(\BVr)$, not the particular bases $\BWr$ and $\BVr$ used in~\eqref{eq:pg_proj}. Thus, $\BWr$ and $\BVr$ are typically replaced by orthonormal matrices to avoid ill-conditioning and singularities in the reduced model.
Note that $\BMkr$ is symmetric for each $k$ by construction.

\subsection{The $\CH_2$ system norm}
\label{ss:h2norm}
Here, we recall the relevant systems theory for \LQO{} systems from~\cite{BenGPD21}.
The nonlinearity in~\eqref{eq:lqosys} is entirely captured by the quadratic term in the output equation; the state equations evolve \emph{linearly} in $\Bx$.
Thus, the time-domain input-to-output map modeled by the \LQO{} system in~\eqref{eq:lqosys} is entirely characterized by \emph{two} Volterra kernels. 
For any input function $\Bu$ and $t \geq 0$, the corresponding output $\By$ of~\eqref{eq:lqosys} can be written as
\begin{align}
    \begin{split}
        \label{eq:lqosys_inputout}
        \By(t)&=\int_0^t {\Bh_1(\tau)}\Bu(t-\tau)\,d\tau \\
        + \int_0^t&\int_0^t {\Bh_2(\tau_1,\tau_2)}\left(\Bu(t-\tau_1) \otimes \Bu(t-\tau_2)\right)\,d\tau_1d\tau_2.
    \end{split}
\end{align}
The Volterra kernels $\Bh_1\colon[0,\infty)\to\Rpm$ and $\Bh_2\colon[0,\infty)\times [0,\infty)\ \to \R^{p\times m^2}$ are defined as
\begin{equation}
\label{eq:kernels}
        \Bh_1(t)\coloneqq\BC e^{\BA t}\BB,~~
        \Bh_2(t_1,t_2)\coloneqq\BM\left(e^{\BA t_1}\BB\otimes e^{\BA t_2}\BB\right),
\end{equation}
where $\BM\in\R^{p\times n^2}$ is defined according to~\eqref{eq:kronqo}.
Here, the univariate kernel $\Bh_1(t)$ and its convolution with $\Bu$ in~\eqref{eq:lqosys_inputout} describe the purely \emph{linear} term in the output; that is, $\By_1(t) = \BC\Bx(t)$. The bivariate kernel $\Bh_2(t_1,t_2)$ and its convolution with $\Bu\otimes \Bu$ in~\eqref{eq:lqosys_inputout} describe the purely \emph{quadratic} term in the output; i.e., $\By_2(t) = \BM\left(\Bx(t)\otimes\Bx(t)\right)$. 
This description of $\Sys$ can be directly obtained from the governing equations~\eqref{eq:lqosys}; see~\cite[Section~4]{BenGPD21} for a derivation.
From the Volterra kernels in~\eqref{eq:lqosys_inputout}, 
an $\CH_2$ inner product and corresponding norm for \LQO{} systems can be defined~\cite[Definition~3.1]{BenGPD21}. 

\begin{definition}
    \label{def:H2norm} 
    Let $\Sys$ and $\Sysred$ be asymptotically stable \LQO{} systems as in~\eqref{eq:lqosys} and~\eqref{eq:lqosys_red} having the kernels $\Bh_1(t)$, $\Bh_{1,r}(t)$ and $\Bh_2(t_1,t_2)$, $\Bh_{2,r}(t_1,t_2)$ defined according to~\eqref{eq:kernels}, respectively.
    The \emph{$\CH_2$ inner product} of $\Sys$ and $\Sysred$ is
    \begin{align}
    \begin{split}
        \label{eq:H2ip}
        \langle\Sys, \Sysred&\rangle_{\CH_2}\coloneqq \int_{0}^\infty \trace\left(\Bh_1(\tau)@ \Bh_{1,r}(\tau)^{\trans}\right) d\tau \\
        &+ \int_{0}^\infty \int_{0}^\infty \trace\left(\Bh_2(\tau_1,\tau_2)@ \Bh_{2,r}(\tau_1,\tau_2)^{\trans} \right) d\tau_1\, d\tau_2,
    \end{split}
    \end{align}
    where $\trace\left(\cdot\right)$ denotes the trace of a matrix.
    The \emph{$\CH_2$ norm} of $\Sys$ is likewise defined according to
    \begin{align}
    \begin{split}
        \label{eq:H2norm}
        \|\Sys\|_{\CH_2}^2\coloneqq& \int_{0}^\infty \|{\Bh_1(\tau)}\|_{\frob}^2 \,d\tau\\
        + \int_{0}^\infty&\int_{0}^\infty \|{\Bh_2(\tau_1,\tau_2)}\|_{\frob}^2\,d\tau_1\,d\tau_2,
    \end{split}
    \end{align}
    where $\|\cdot\|_{\frob}$ denotes the matrix Frobenius norm.
\end{definition}
If $\BM_k$ for $k=1,\ldots,p$ in~\eqref{eq:lqosys} are identically zero, then $\Sys=\Sysonlyc$ is purely an \LTI{} system with $\Bh_2=\Bzero_{p\times m^2}$, and so~\eqref{eq:H2ip} and~\eqref{eq:H2norm} agree with the usual $\CH_2$ inner product and norm defined for linear dynamical systems; cf.~\cite[Section~5.1]{Ant05}.
Lastly, it follows from \Cref{def:H2norm} that the $\CH_2$ norm of an asymptotically stable \LQO{} system is finite.

Next, we consider an alternative and more computationally tractable characterization of the $\CH_2$ system norm in~\Cref{def:H2norm} based on the \emph{Gramians} of an \LQO{} system.
These formulations will be paramount in deriving gradients of the squared $\CH_2$ system error and associated $\CH_2$-optimality conditions in~\Cref{sec:h2_opt}.

\subsubsection{The quadratic-output observability Gramian}

Because the nonlinearity of $\Sys$ in~\eqref{eq:lqosys} is limited to the output equation, its input-to-state map is identical to that of the related \LTI{} system $\Sysonlyc$ having the same state, input, and linear output matrices, $\BA\in\Rnn$, $\BB\in\Rn$, and $\BC\in\Rpn$, with $\BM_k=\Bzero_{n\times n}$ for all $k$.
Thus, the \emph{reachability Gramian} $\BP\in\Rnn$ of the \LQO{} system $\Sys$ is the same as the classical reachability Gramian defined for purely \LTI{} systems~\cite[Section~4.3]{Ant05}, i.e.,
\begin{align}
    \label{eq:reach_gram}
    \BP =\int_0^\infty e^{\BA \tau}\BB \left(e^{\BA \tau}\BB\right)^{\trans} \,d\tau.
\end{align}
Note that none of the output matrices $\BC$ or $\BM_k$ play a role in defining $\BP$.
Under the assumption that $\Sys$ is asymptotically stable, $\BP$ is unique~\cite[Prop.~6.2]{Ant05}, and can be computed as the solution of the Lyapunov equation
\begin{align}
    \label{eq:reach_lyap}
    \BA \BP + \BP \BA^{\trans} + \BB@\BB^{\trans}=\Bzero.
\end{align}
An algebraic \emph{quadratic-output observability Gramian} based on the nonlinear state-to-output map of an \LQO{} system was introduced in~\cite{BenGPD21}. 
In a sense, the \QO{} observability Gramian generalizes the classical observability Gramian of an \LTI{} system~\cite[Section~4.3]{Ant05}. 
Consider an \LQO{} system $\Sys$ as in~\eqref{eq:lqosys} and define the intermediate matrices $\BQ_1\in\Rnn$ and $\BQtwok\in\Rnn$ by
\begin{align}
    \label{eq:Q1}
    \BQ_1 &\coloneqq\int_0^\infty e^{\BA^{\trans} \tau}\BC^{\trans} \left(e^{\BA ^{\trans}\tau}\BC^{\trans}\right)^{\trans} \,d\tau\\
    \begin{split}
    \label{eq:Q2k}
    \mbox{and}~~\BQtwok&\coloneqq\int_0^\infty\int_0^\infty e^{\BA^{\trans} \tau_1}\BM_k e^{\BA \tau_2}\BB   \times \\
    & \ \ \ \ \ \ \left(e^{\BA^{\trans} \tau_1}\BM_k e^{\BA \tau_2}\BB\right)^{\trans} d\tau_1\,d\tau_2,
    \end{split}
\end{align}
for each $k=1,\ldots p$. 
(Note that $\BQ_1\in\Rnn$ is just the classical observability Gramian of the linear system $\Sysonlyc$.)
Then the \QO{} observability Gramian $\BQ\in \Rnn$ of $\Sys$ is defined as
\begin{align}
\begin{split}
    \label{eq:QO_obsv_gram}
    \BQ \coloneqq\BQ_1+\sum_{k=1}^p\BQtwok.
\end{split}
\end{align}
Akin to the linear system Gramians, it is proven in~\cite[Section~4.1]{BenGPD21} that if $\Sys$ is asymptotically stable, then $\BQ$ can be computed as the unique solution to the Lyapunov equation
\begin{align}
    \label{eq:QO_obsv_lyap}
    \BA^{\trans} \BQ + \BQ \BA + \BC^{\trans}\BC + \sum_{k=1}^p\BM_k \BP \BM_k=\Bzero, 
\end{align}
where $\BP$ is the reachability Gramian of $\Sys$ according to~\eqref{eq:reach_gram}. 
See~\cite[Section~4.1]{BenGPD21} for a detailed derivation of $\BQ$. 
Based on the energy functionals defined by the Gramians in~\eqref{eq:reach_gram} and~\eqref{eq:QO_obsv_gram}, a generalization of the balanced truncation model reduction algorithm for \LQO{} systems~\eqref{eq:lqosys} was proposed in~\cite{BenGPD21}.

It was proven in~\cite[Proposition~3.3]{BenGPD21} that the \QO{} observability Gramian $\BQ$ in~\eqref{eq:QO_obsv_gram} can be used to compute the $\CH_2$ norm~\eqref{eq:H2norm} of an \LQO{} system.
The $\CH_2$ inner product of two \LQO{} systems~\eqref{eq:H2ip} can similarly be computed using the solution to a pair of Sylvester equations.
Next, we provide a proof of this result for \LQO{} systems with multiple quadratic outputs and also show that the $\CH_2$ inner product and norm for \LQO{} systems in~\Cref{def:H2norm} can be calculated from the reachability Gramian $\BP$ in~\eqref{eq:reach_gram} and quadratic-output matrices $\BM_k$. 
The result~\cite[Prop.~3.3]{BenGPD21} proves the formulae~\eqref{eq:H2ip_Q} and~\eqref{eq:H2norm_Q} in~\Cref{thm:H2_from_gramians} for the special case of a single quadratic output with $\BC=\Bzero_{p\times n}$.
A related formula was independently proven in the recent work~\cite[Lemma~5.2]{PrzDGB24} for systems of differential-algebraic equations with quadratic-output functions.

\begin{theorem}
    \label{thm:H2_from_gramians}
    Let $\Sys$ and $\Sysred$ be asymptotically stable \LQO{} systems as in~\eqref{eq:lqosys} and~\eqref{eq:lqosys_red}, respectively. Let $\BX\in\Rnr$ and $\BZ\in\Rnr$ denote the unique solutions to the Sylvester equations
    \begin{align}
        \label{eq:sylveq_X}
        \BA\BX + \BX\BAr^{\trans} + \BB\BBr^{\trans}&=\Bzero\\ 
        \label{eq:sylveq_Z}
        \mbox{and}~~\BA^{\trans}\BZ + \BZ \BAr - \sum_{k=1}^p\BM_k\BX\BMkr - \BC^{\trans}\BCr&=\Bzero.
    \end{align}
    Then, the $\CH_2$ inner product of $\Sys$ and $\Sysred$ is given by
    \begin{align}
        \label{eq:H2ip_Q}
        \langle\Sys, \Sysred\rangle_{\CH_2}&= -\trace\left(\BB^{\trans}\BZ\BBr\right)\\  
        \label{eq:H2ip_P}
        &=\trace\left(\BC\BX\BCr^{\trans}\right)+\sum_{k=1}^p{\trace\left(\BX^{\trans}\BM_k\BX\BMkr\right)}. 
    \end{align}
    If $\Sysred=\Sys$, then $\BX = \BP\in\Rnn$ and $\BZ=\BQ\in\Rnn$ according to~\eqref{eq:reach_gram} and~\eqref{eq:QO_obsv_gram}. Thus, the $\CH_2$ norm of $\Sys$ is given by
    \begin{align}
        \nonumber
        \|\Sys\|_{\CH_2}^2 &=\trace\left(\BB^{\trans}\BQ_1\BB\right) + \sum_{k=1}^p \trace\left(\BB^{\trans}\BQtwok\BB\right)\\
        &=\trace\left(\BB^{\trans}\BQ\BB\right)        \label{eq:H2norm_Q}\\
        \label{eq:H2norm_P}
        &=\trace\left(\BC\BP\BC^{\trans}\right) + \sum_{k=1}^p\trace\left(\BP\BM_k\BP\BM_k\right).
    \end{align}
\end{theorem}

\begin{proof}
    Throughout this proof, take $\Bh_1(t)$, $\Bh_2(t_1,t_2)$ and $\Bh_{1,r}(t)$, $\Bh_{2,r}(t_1,t_2)$ to denote the Volterra kernels of $\Sys$ and $\Sysred$ respectively according to~\eqref{eq:kernels}.
    Consider the Sylvester equation
    \begin{align}
    \label{eq:sylveq_Z1}
        \BA^{\trans}\BZ_1+\BZ_1\BAr-\BC^{\trans}\BCr=\Bzero.
    \end{align}
    Because $\Sys$ and $\Sysred$ are asymptotically stable, the spectra of $\BA$ and $-\BAr$ are disjoint. Thus, the Sylvester equations~\eqref{eq:sylveq_X} and~\eqref{eq:sylveq_Z1} have unique solutions~\cite[Prop.~6.2]{Ant05}. These solutions can be explicitly written as
    \begin{align*}
        \BX &= \int_0^\infty e^{\BA\tau}\BB\left( e^{\BAr\tau}\BBr\right)^{\trans}@d\tau,\\
        \BZ_1&=-\int_0^\infty e^{\BA^{\trans} \tau}\BC^{\trans} \left(e^{\BAr^{\trans}  \tau}\BCr^{\trans}\right)^{\trans} @d\tau.
    \end{align*}
    By the same argument, for each $k=1,\ldots,p$, the Sylvester equation
    \begin{align}
        \label{eq:sylveq_Z2k}
        \BA^{\trans}\BZ_k+\BZ_k\BAr-\BM_k\BX\BMkr=\Bzero
    \end{align}
    has a unique solution of the form 
    \begin{align*}
        \BZtwok =-  \int_0^\infty e^{\BA^{\trans}\tau}\BM_k\BX\left(e^{\BAr^{\trans}\tau}\BMkr\right)^{\trans}@d\tau.
    \end{align*}
    Summing up the equations~\eqref{eq:sylveq_Z1}, and~\eqref{eq:sylveq_Z2k} over all $k$, yields~\eqref{eq:sylveq_Z}. By uniqueness, the solution $\BZ\in\Rnr$ to~\eqref{eq:sylveq_Z} is such that 
    \begin{equation*}
        \BZ=\BZ_1 + \sum_{k=1}^p\BZtwok.
    \end{equation*}
    Substituting the integral form of $\BX$ into $\BZtwok$ for each $k$ and moving terms inside the integrand, $\BZ$ can be expressed as
    \begin{align*}
        \BZ = -\int_0^\infty e^{\BA^{\trans} \tau}\BC^{\trans}& \left(e^{\BAr^{\trans}  \tau}\BCr^{\trans}\right)^{\trans} @d\tau\\
        - \sum_{k=1}^p\int_0^\infty &\int_0^\infty \left( e^{\BA^{\trans}\tau_1}\BM_k e^{\BA\tau_2}\BB\right)\times\\
        &\left( e^{\BAr^{\trans}\tau_1}\BMkr e^{\BAr\tau_2}\BBr\right)^{\trans}@d\tau_2@d\tau_1.
    \end{align*}
    Then, by the invariance of the trace under cyclic permutation, we have
    \begin{align*}
        \trace\left(\BB^{\trans}\BZ\BBr\right) = -\int_0^\infty \trace@\bigg(\BC e^{\BA \tau}\BB& \big(\BCr e^{\BAr  \tau}\BBr\big)^{\trans}\bigg) @d\tau\\
        - \sum_{k=1}^p\int_0^\infty\int_0^\infty \trace@\bigg( \BB^{\trans} e^{\BA^{\trans}\tau_1} & \BM_k e^{\BA\tau_2}\BB\times\\
        \big(\BBr^{\trans} e^{\BAr^{\trans}\tau_1}\BMkr e^{\BAr\tau_2}&\BBr\big)^{\trans}\bigg)@d\tau_2@d\tau_1.
    \end{align*}
    Note that the first term in the expression for $\trace\left(\BB^{\trans}\BZ\BBr\right)$ above is precisely the first term in the $\CH_2$ inner product~\eqref{eq:H2ip}.
    It remains to be shown that the remaining terms in the summand over $k$ equate to the second term in~\eqref{eq:H2ip}.
    To show this, we observe that by construction of $\BM$ in~\eqref{eq:kronqo} and properties of the Kronecker product~\cite{Bre78}, the bivariate kernel $\Bh_2(t_1,t_2) = \BM\left(e^{\BA t_1}\BB\otimes e^{\BA t_2}\BB\right)$ can be expressed as
    \begin{align*}
        \Bh_2(t_1,t_2)= \begin{bmatrix}
           \vecm\big( \BB^{\trans} e^{\BA^{\trans} t_1}\BM_1 e^{\BA t_2}\BB\big)^{\trans}\\
            \vdots\\
            \vecm\big(\BB^{\trans} e^{\BA^{\trans} t_1}\BM_p e^{\BA t_2}\BB\big)^{\trans}
        \end{bmatrix},
    \end{align*}
    and likewise for $\Bh_{2,r}(t_1,t_2)$.
    Thus, for all $t_1,t_2\geq 0$
    \begin{align*}
        \trace\left(\Bh_2(t_1,t_2)@ \Bh_{2,r}(t_1,t_2)^{\trans}\right)&\\
        =@\sum_{k=1}^p\trace\bigg(\BB^{\trans} e^{\BA^{\trans} t_1}\BM_k & e^{\BA t_2}\BB \big(\BBr^{\trans} e^{\BAr^{\trans} t_1}\BMkr e^{\BAr t_2}\BBr\big)^{\trans}\bigg),
    \end{align*}
    because the Frobenius inner product of two matrices equates to the vector inner product of their vectorized forms~\cite[Ch.~12.3]{GolV13}.
    Integrating both sides of the above equality yields
        \begin{align*}
        \int_0^\infty\int_0^\infty \trace\left(\Bh_2(\tau_1,\tau_2)^{\vphantom{\trans}}@ \Bh_{2,r}(\tau_1,\right.&\left.\tau_2)^{\trans}\right)\,d\tau_1\,d\tau_2\\
       = - \sum_{k=1}^p\int_0^\infty\int_0^\infty \trace@\left( \BB^{\trans} e^{\BA^{\trans}\tau_1} \right.& \BM_k e^{\BA\tau_2}\BB\times\\
       \big(\BBr e^{\BAr^{\trans}\tau_1}\BMkr e^{\BAr\tau_2}& \left.\BBr\big)^{\trans}\right)\,d\tau_2\,d\tau_1,
    \end{align*}
    proving that $\langle\Sys,@\Sysred\rangle_{\CH_2}=-\trace\left(\BB^{\trans}\BZ\BBr\right),$ as claimed. The formula for the $\CH_2$ norm in~\eqref{eq:H2norm_Q} follows directly by replacing $\Sys$ with $\Sysred$.
    
    The proof of~\eqref{eq:H2ip_P} now follows straightforwardly from~\eqref{eq:H2ip_Q}.
    First note that the matrix equations~\eqref{eq:sylveq_X} and~\eqref{eq:sylveq_Z} can be reformulated as equivalent linear systems
    \begin{align*}
        \left(\BI_n \otimes \BA + \BAr\otimes \BI_r\right) \vecm\left(\BX\right)&=-\vecm\left(\BB\BBr^{\trans}\right),\\
        \left(\BI_n \otimes \BA^{\trans} + \BAr^{\trans} \otimes \BI_r\right) \vecm\left(\BZ\right)&=\vecm\left(\BC^{\trans}\BCr\right)\\+\sum_{k=1}^p& \left(\BM\otimes\BMr\right)\vecm\left(\BX\right).
    \end{align*}
    Using the characterization of $\langle\Sys,@\Sysred\rangle_{\CH_2}$ in~\eqref{eq:H2ip_Q} as well as properties of the trace and Kronecker product, it follows that 
    \begin{align*}
          \langle\Sys,@\Sysred\rangle_{\CH_2}&=-\trace\left(\BB^{\trans}\BZ\BBr\right)
         ={-\vecm\left(\BB\BBr^{\trans}\right)^{\trans}\vecm\left(\BZ\right)}\\
         &=\vecm\left(\BX\right)^{\trans}\left(\BI_n \otimes \BA^{\trans} + \BAr^{\trans}\otimes \BI_r\right) \vecm\left(\BZ\right)\\
         &= \vecm\left(\BX\right)^{\trans}
        \left(\vecm\left(\BC^{\trans}\BCr\right)+
        \sum_{k=1}^p  \left(\BM\otimes\BMr\right)\vecm\left(\BX\right)\right)\\
        &=\trace\left(\BC\BX\BCr^{\trans}\right)+\sum_{k=1}^p\trace\left(\BX^{\trans}\BM_k\BX\BMkr\right).
     \end{align*}
     This proves~\eqref{eq:H2ip_P}; the analogous formula for the norm~\eqref{eq:H2norm_P} follows from replacing $\Sysred$ with $\Sys$.
\end{proof}

\subsubsection{Relating the $\CH_2$ system error and $\CL_\infty$ output error}
We take a moment here to motivate the choice of the $\CH_2$ norm as a performance metric in the model reduction of \LQO{} systems by revising the error bound from~\cite[Theorem~3.4]{BenGPD21}.
Throughout, we take $\CL_{2}^{n_1\times n_2}$ to denote the space of $n_1\times n_2$ matrix-valued functions $\BG\colon\Omega \to \R^{n_1\times n_2}$ that satisfy the square integrability condition
\begin{equation}
\label{eq:LpNorm}
    \|\BG\|_{\CL_2^{n_1\times n_2}}\coloneqq\left(\int_{\Omega}\|\BG\|_{\frob}^2@d@\omega\right)^{1/2}<\infty.
\end{equation}
In any case, $\Omega$ will either be $[0,\infty)$ or $[0,\infty)\times[0,\infty)$. If $n_2=1$, this is a vector-valued norm.
Consider the full- and reduced-order systems $\Sys$ and $\Sysred$ in~\eqref{eq:lqosys} and~\eqref{eq:lqosys_red}, respectively.
Recall that an effective surrogate model should replicate the full quantity of interest $\Byr\approx\By$ for a variety of external inputs $\Bu$.
Suppose that one wishes to design a reduced model so that the error due to the approximate output is uniformly small over all values $t>0$. In other words, the $\CL_\infty$ output error $\|\By-\Byr\|_{\CL^p_\infty}\coloneqq\sup_{t\geq0}\|\By(t)-\Byr(t)\|_\infty$
should be small for admissible $\Bu$.
Following~\eqref{eq:lqosys_inputout}, the output error at any time $t>0$ may be expressed as
\begin{align*}
    \By(t)-\Byr(t)
        &=\int_0^t \left(\Bh_1(\tau)-\Bh_{1,r}(\tau)\right)\Bu(t-\tau)\,d\tau \\
        +\int_0^t\int_0^t &\left(\Bh_2(\tau_1,\tau_2)- \Bh_{2,r}(\tau_1,\tau_2)\right)\times\\
        &~~~~~~\left(\Bu(t-\tau_1)\otimes\Bu(t-\tau_2) \right)\,d\tau_1\,d\tau_2.
\end{align*} 
Applying the vector norm $\|\cdot\|_\infty\colon\Rp\to\R$ to both sides of this equality, we obtain
\begin{align*}
    \|\By(t)-\Byr(t)\|_\infty
    &\leq \int_0^t \|\Bh_1(\tau)-\Bh_{1,r}(\tau)\|_{\frob}\|\Bu(t-\tau)\|_2@d\tau \\
    +\int_0^t\int_0^t &\|\Bh_2(\tau_1,\tau_2)- \Bh_{2,r}(\tau_1,\tau_2)\|_{\frob}\times\\
    &~~~~~~\|\Bu(t-\tau_1)\otimes\Bu(t-\tau_2)\|_2@d\tau_1@d\tau_2,
\end{align*}
where the inequality follows from first applying the integral triangle inequality, and subsequently from applying the inequalities $\|\Bv\|_\infty\leq \|\Bv\|_2$ and $\|\BH\Bv\|_2\leq\|\BH\|_2\|\Bv\|_2\leq\|\BH\|_{\frob}\|\Bv\|_2$ for any $\BH\in\C^{n_1\times n_2}$ and $\Bv\in\C^{n_2}$ to the integrands.
Because each integrand above is strictly nonnegative, we may take the upper integral limits $t\to\infty$. 
Applying the Cauchy-Schwarz inequality to the relevant $\CL_2$ inner products of the norms $\|\Bh_1-\Bh_{1,r}\|_{\frob}$ and $\|\Bu\|$, as well as $\|\Bh_2-\Bh_{2,r}\|_{\frob}$ and $\|\Bu\otimes \Bu\|_2$, allows us to simplify the bound further:
\begin{align*}
    \|\By(t)-\Byr(t)\|_\infty
    &\leq \|\Bh_1-\Bh_{1,r}\|_{\CL_2^{p\times m}}\|\Bu\|_{\CL^m_2} \\
    +\|\Bh_2&-\Bh_{2,r}\|_{\CL_2^{p\times m^2}}\|\Bu\otimes \Bu\|_{\CL^{m^2}_2},
\end{align*}
where the $\CL_2$ norms are defined as in~\eqref{eq:LpNorm}.
Applying the Cauchy-Schwarz inequality again, this time to the $2$-vectors containing the individual kernel errors, and squaring both sides, we retrieve a bound for any time $t>0$ involving the $\CH_2$ system error:
\begin{align*}
    \|\By(t)-\Byr(t)\|_\infty^2\leq& \bigg(\underbrace{\|\Bh_1-\Bh_{1,r}\|_{\CL_2^{p\times m}}^2 + \|\Bh_2-\Bh_{2,r}\|_{\CL_2^{p\times m^2}}^2}_{=\|\Sys-\Sysred\|_{\CH_2}^2}\bigg)\\
    &~~~~~\times \left(\|\Bu\|_{\CL^m_2}^2+\|\Bu\otimes \Bu\|_{\CL^{m^2}_2}^2\right).
\end{align*}
Because $t>0$ is arbitrarily specified, taking the supremum over all time reveals that the \emph{$\CL_\infty$ output error is bounded above by the $\CH_2$ linear quadratic-output system error}
\begin{align}
    \label{eq:H2bound}
    \|\By-\Byr\|_{\CL^p_\infty}^2 \leq \|\Sys-\Sysred\|_{\CH_2}^2 \left(\|\Bu\|_{\CL^m_2}^2 + \|\Bu\otimes \Bu\|_{\CL^{m^2}_2}^2\right).
\end{align}
We emphasize that the upper bound~\eqref{eq:H2bound} holds \emph{uniformly} in time.
Thus, for an admissible input $\Bu$, a small $\CH_2$ error guarantees that the reduced-model output $\By_r$ is a 
uniformly good approximation to the original output $\By$. 

The bound in~\eqref{eq:H2bound} motivates using the $\CH_2$ error as a performance measure in \LQO-\MOR{}. 
Indeed, if one wants the `worst case' error in the output $\|\By-\Byr\|_{\CL_\infty}$ to be small, then one should seek a reduced model $\Sysred$ so that the $\CH_2$ error $\|\Sys-\Sysred\|_{\CH_2}$ appearing in~\eqref{eq:H2bound} is as small as possible. 
Motivated by this, we consider the \emph{$\CH_2$-optimal model reduction problem for \textsf{LQO} systems}: Given the full-order asymptotically stable \LQO{} system $\Sys$ in~\eqref{eq:lqosys}, we seek a reduced-order, and also asymptotically stable, \LQO{} reduced model $\Sysred$ represented in~\eqref{eq:lqosys_red} such that the $\CH_2$ norm of the error system is \emph{minimized}:
\begin{align}
    \label{eq:H2opt_lqomor}
    \Sysred= \argmin_{\substack{\wh\Sys_r \mathrm{~ stable}\\ \dim(\wh\Sys_r)=r}} \CJ(\wh\Sys_r), \ \ \CJ(\wh\Sys_r)\coloneqq\|\Sys-\wh\Sys_r\|_{\CH_2}^2.
\end{align} 
The squared $\CH_2$ error in~\eqref{eq:H2opt_lqomor} is used solely for the ease of calculating gradients of $\CJ$ in~\Cref{sec:h2_opt}.
By the asymptotic stability assumption imposed on the full- and reduced-order models, the corresponding $\CH_2$ error is guaranteed to be finite. 
The minimization problem in~\eqref{eq:H2opt_lqomor} is, in general, non-convex, and the characterization of global minimizers is elusive.
Instead, we wish to identify \emph{local minimizers} of the $\CH_2$ error in~\eqref{eq:H2opt_lqomor} that satisfy \FONC{}s for $\CH_2$ optimality.

Before presenting our major theoretical results in the next section, we establish notation pertaining to gradients of functions defined over normed vector spaces; our presentation follows that of~\cite{Col12}.
Consider a Fr\'{e}chet differentiable function $f\colon U \to \R$ defined on an open subset $U$ of a Hilbert space $X$ endowed with the inner product $\langle\cdot,\cdot\rangle_X\colon X\times X\to\R$.
For any $x_0\in U$, the \emph{gradient} of $f$ at $x_0$ is the unique element $\nabla f (x_0)\in X$ so that
\begin{align}
 \label{eq:deriv_as_grad}
 f(x_0+h)=f(x_0) + \langle\nabla f (x_0),@ h\rangle_X + O(\|h\|_X^2),
\end{align}
for all $h$ in a neighborhood of zero.
We write $g(x)=O(\|h\|_X^2)$ if $\lim_{h\rightarrow 0}\frac{g(h)}{\|h\|_X}=0$.
If $\nabla f (x_0)=0$, we call $x_0\in X$ a \emph{critical} point of $f$. If $f$ has a local extremum at a point $x_0$, then necessarily $x_0$ is a critical point~\cite[Cor.~2.5]{Col12}.
For a multivariate function $f\colon X_1\times\ldots\times X_\ell\to\R$, partial gradients $\nabla_{x_i} f(x_1,\ldots,x_\ell)$ are defined analogously.

\section{Optimal $\CH_2$ model reduction}
\label{sec:h2_opt}
This section contains the main theoretical results of the paper.~\Cref{thm:gradients} presents gradients of the squared $\CH_2$ \LQO{} system error in~\eqref{eq:H2opt_lqomor} with respect to the reduced-order matrices of the \LQO-\ROM{} in~\eqref{eq:lqosys_red} as parameters. 
The stationary points of these gradients automatically yield Gramian-based \FONC{}s for the $\CH_2$-optimal model reduction of \LQO{} systems, which we present in~\Cref{thm:foncs_gramians}. 
These provide a set of structured $\CH_2$-optimality conditions for the \LQO-\MOR{} problem.
To set the stage for these results and make comparisons later on, we first review the Gramian-based $\CH_2$-optimality conditions for linear model reduction due to Wilson~\cite{Wil70}.

\subsection{Optimal $\CH_2$ model reduction of purely linear systems}
\label{ss:h2_opt_lti}
For the discussion in this subsection, we restrict our attention to \emph{purely} \LTI{} dynamical systems. Consider
\begin{align}\label{eq:ltisys}
\Sysonlyc: \begin{cases}  ~@@\dot\Bx(t)=\BA\Bx(t)+\BB \Bu(t),\quad \Bx(0)=\Bzero,\\
\Byone(t)=\BC\Bx(t).
\end{cases}
\end{align} 
The state, input, and output dimensions are the same as in~\eqref{eq:lqosys}.
Analogous to the \LQO{} model reduction problem, we seek a linear reduced model of the form
\begin{align}\label{eq:ltisys_red}
    \Sysonlycred: \begin{cases}  ~~@@\dot{\Bx}_r(t)=\BAr\Bxr(t)+\BBr \Bu(t),\quad \Bxr(0)=\Bzero,\\
    \hspace{0.5mm} \Byoner(t)=\BCr\Bxr(t).
    \end{cases}
\end{align}
It will be fruitful to view the \LTI{} system in~\eqref{eq:ltisys} as a special case of the \LQO{} system class~\eqref{eq:ltisys} having the realization $\left(\BA, \BB, \BC, \Bzero_{n\times n}, \ldots, \Bzero_{n\times n}\right)$.
With this perspective, much of the systems theory introduced in~\Cref{sec:background} reduces to the analogous \LTI{} systems theory.
The input-to-output map of the \LTI{} system $\Sysonlyc$ in~\eqref{eq:ltisys} is fully realized by the one-dimensional kernel $\Bh_1(t)=\BC e^{\BA t}\BB$, and $\Bh_2(t_1,t_2)=\Bzero_{p\times m^2}$ in this context.
As a consequence, the $\CH_2$ norm of $\Sysonlyc$ as given in~\Cref{def:H2norm} agrees with the linear $\CH_2$ system norm; cf.~\cite[Section~5.1]{Ant05}. 
The \QO{} observability Gramian of $\Sysonlyc$ becomes the classical {observability Gramian} of a linear system, that is, $\BQ_1\in\Rnn$ in~\eqref{eq:Q1}, since $\BQ_{2,k}=\Bzero_{n\times n}$ for all $k$.
The Lyapunov equation~\eqref{eq:QO_obsv_lyap} reduces to
\begin{align}
\label{eq:obsv_lyap}
    \BA^{\trans} \BQ_1 + \BQ_1 \BA + \BC^{\trans}\BC =\Bzero.
\end{align}
(When discussing the observability Gramian of a purely linear system $\Sysonlyc$, we keep the subscript to clearly differentiate from the general \QO{} observability Gramian in~\eqref{eq:QO_obsv_gram}.)
The reachability Gramian $\BP$ in~\eqref{eq:reach_gram} is unchanged.
A similar bound to~\eqref{eq:H2bound} can be derived; see~\cite[Sec.~2.1.1]{AntBG20}.
The Gramian-based framework for linear $\CH_2$-optimal model reduction is attributed to Wilson~\cite{Wil70} for multi-input, multi-output systems.
These results were later shown to be equivalent to interpolation-based optimality conditions~\cite{MeiL67, VanGPA08, GugAB08}.
The starting point for deriving the Gramian-based optimality conditions in~\cite{Wil70, VanGPA08} is expressing the $\CH_2$ norm as a function of the \emph{error system} $\Sysonlyc -\Sysonlycred$, which is itself an order-$(n+r)$ linear system of the form~\eqref{eq:ltisys}. A state-space realization of $\Sysonlyc -\Sysonlycred$ is given by $(\BAe,\BBe,\BCe)$, where $\BAe\in\R^{(n+r)\times(n+r)}$, $\BBe\in\R^{(n+r)\times m}$, and $\BCe\in\R^{p\times (n+r)}$ are given by
\begin{align}\label{eq:ltisys_error_sys}
    \BAe=\begin{bmatrix}
        \BA & \\
        & \BAr\\
    \end{bmatrix}, \ {\BBe = \begin{bmatrix}
        \BB\hphantom{_r} \\  \BBr 
    \end{bmatrix}},\ \BCe = \begin{bmatrix}
        \BC & -\BCr
    \end{bmatrix}.
\end{align}
The internal state and output of the error system are given by $\Bxe=\begin{bmatrix}
    \Bx^{\trans}&
    \Bxr^{\trans}
\end{bmatrix}^{\trans}$ and $\Byone-\Byoner$. Take $\BPe, \BQonee\in\R^{(n+r)\times(n+r)}$ to denote the reachability and observability Gramians of~\eqref{eq:ltisys_error_sys} which solve the Lyapunov equations~\eqref{eq:reach_lyap} and~\eqref{eq:obsv_lyap}.
These are expressed in $2\times 2$ block form as
\begin{align}
\label{eq:linear_err_gramians}
    \BPe=\begin{bmatrix}
        \BP & \BX@@@@@\\
        @@@@@\BX^{\trans} & \BPr
    \end{bmatrix}
    \quad\mbox{and}\quad    
    \BQonee=\begin{bmatrix}
        @@@@@\BQ_1 & \BZ_1@@@@@\\
        @@@@@\BZ_1^{\trans} & \BQoner
    \end{bmatrix},
\end{align}
where $\BP,\BQ_1\in\Rnn$ and $\BPr,\BQoner\in\Rrr$ are the reachability, observability Gramians of the full- and reduced-order linear models in~\eqref{eq:ltisys} and~\eqref{eq:ltisys_red}, respectively, while $\BX\in\Rnr$ and $\BZ_1\in\Rnr$ solve the matrix equations
\begin{align}
\begin{split}
\label{eq:linear_sylv}
    \BA\BX + \BX\BAr^{\trans} + \BB\BBr^{\trans}&=\Bzero,\\
    \mbox{and}~~\BA^{\trans}\BZ_1 + \BZ_1 \BAr -\BC^{\trans}\BCr&=\Bzero.
\end{split}
\end{align}
These equalities~\eqref{eq:linear_sylv} can be deduced by solving for $\BX$ and $\BZ_1$ directly from equations~\eqref{eq:reach_lyap} and~\eqref{eq:obsv_lyap} applied to the error system. 
Clearly, the $\CH_2$ norm of the error system~\eqref{eq:ltisys_error_sys} is precisely the approximation error induced by $\Sysonlycred$.
Applying~\Cref{thm:H2_from_gramians} to $\Sysonlyc-\Sysonlycred$, the squared $\CH_2$ system error can be written as
\[\CJ(\Sysonlycred)=\|\Sysonlyc-\Sysonlycred\|_{\CH_2}^2=\trace\left(\BBe^{\trans}\BQonee\BBe\right)=\trace\left(\BCe\BPe\BCe^{\trans}\right).\]
From this expression,~\cite[Theorem~3.3]{VanGPA08},~\cite{Wil70} show that gradients of $\CJ$ with respect to 
$\BAr$, $\BBr$, and $\BCr$  
are given by
\begin{align}
    \begin{split}
        \label{eq:gradients_lti}
        \gradA&=2\left(\BPr\BQoner+\BX^{\trans}\BZ_1\right),\\
        \gradB&=2\left(\BQoner\BBr+\BZ_1^{\trans}\BB\right),\\
        \gradC&=2\left(\BCr\BPr-\BC\BX\right).
    \end{split}
\end{align}
(We drop the dependence of $\CJ$ on $\Sysonlycred$ for convenience of notation.)
The matrices $\BX\in\Rnr$ and $\BZ_1\in\Rnr$ appearing in~\eqref{eq:gradients_lti} are those that satisfy~\eqref{eq:linear_sylv}.
If the reduced \LTI{} system $\Sysonlycred$ in~\eqref{eq:ltisys_red} is a local minimizer of the $\CH_2$ error, then the gradients in~\eqref{eq:gradients_lti} are identically zero, and the reduced order matrices $\BAr,\BBr,$ and $\BCr$ satisfy the identities: 
\begin{align}
\begin{split}
    \label{eq:wilson_lti}
       \Bzero&= \BPr\BQoner+ \BX^{\trans}\BZ_1,\\
       \Bzero &=\BQoner\BBr+ \BZ_1^{\trans}\BB,\\
        \Bzero &=\BCr\BPr-\BC\BX.
\end{split}
\end{align}
These are the \emph{Gramian-based}, or \emph{Wilson} conditions for the $\CH_2$-optimal model reduction of \LTI{} systems.
Under the assumption that $\BPr$ and $\BQoner$ are nonsingular, if $\Sysonlycred$ is $\CH_2$ optimal then it is realized in a Petrov-Galerkin framework~\cite{Wil70},~\cite[Theorem~3.4]{VanGPA08}, where the projection matrices in~\eqref{eq:pg_proj} are given by $ \BVr=\BX\BPr^{-1}$ and $\BWr=-\BZ_1\BQoner^{-1}.$
Evidently, $\BVr$ and $\BWr$ depend explicitly upon the $\CH_2$-optimal reduced model $\Sysonlycred$, which of course is not known \emph{a priori.}
To handle this issue, the \emph{two-sided iteration algorithm} $(\TSIA{})$ based on iterative projection using the solutions to the pair of Sylvester equations in~\eqref{eq:linear_sylv} was proposed in~\cite{XuZ11}. If \TSIA{} converges, it yields a \LTI{}-\ROM{} that satisfies~\eqref{eq:wilson_lti}.
Practical improvements that make \TSIA{} more computationally efficient were considered in~\cite{BenKS11}.

\subsection{Optimal $\CH_2$ model reduction of \textsf{LQO} systems}

We return our attention to the $\CH_2$-optimal model reduction problem for \LQO{} systems described in~\Cref{ss:h2norm}.
Recall the minimization problem~\eqref{eq:H2opt_lqomor}, and the associated cost function
\[\CJ(\Sysred)=\CJ \left(\BAr,\BBr,\BCr,\BMoner,\ldots,\BMkr\right)=\|\Sys-\Sysred\|_{\CH_2}^2.\]
We view the objective function $\CJ\colon\Rrr\times\Rrm\times\Rpn\times\cdots\times\Rrr\to\R$ as taking the reduced-order matrices that determine $\Sysred$ in~\eqref{eq:lqosys_red} as arguments.
That is, $\CJ$ is a multivariate function defined over a real-valued Hilbert space of matrices in each variable.
Per~\cite[Cor.~2.5]{Col12}, if the reduced \LQO{} system $\Sysred$ in~\eqref{eq:lqosys_red} having the realization $$\left(\BAr,\BBr,\BCr,\BMoner,\ldots,\BMkr\right)$$ is a local minimizer of the $\CH_2$ error, then the partial gradients of $\CJ$ with respect to the reduced-order matrices $\BAr$, $\BBr$, $\BCr$, and $\BMkr$ are necessarily zero. 
Thus, computing gradients of $\CJ$ in~\eqref{eq:H2opt_lqomor} and setting them to zero paves the way for developing the Wilson optimality conditions in the \LQO{} setting.

Following suit with the linear problem, our starting point for computing gradients of $\CJ$ is expressing the objective function in terms of the error system $\Sys-\Sysred$. The matrices $\left(\BAe,\BBe,\BCe,\BMonee,\ldots,\BMpe\right)$, where $\BAe\in\R^{(n+r)\times(n+r)}$, $\BBe\in\R^{(n+r)\times m}$, $\BCe\in\R^{p\times (n+r)}$ and $\BMke\in\R^{(n+r)\times(n+r)}$ are defined as
\begin{align}
\begin{split}
\label{eq:lqosys_error_sys}
    \BAe&=\begin{bmatrix}
        \BA & \\
        & \BAr\\
    \end{bmatrix}, \ {\BBe = \begin{bmatrix}
        \BB\hphantom{_r} \\ \BBr
    \end{bmatrix}},\ \BCe = \begin{bmatrix}
        \BC & -\BCr
    \end{bmatrix},\\
    \BMke&=\begin{bmatrix}
    \BM_k & \\ & -\BMkr
\end{bmatrix}, \ \mbox{for each} \ k=1,\ldots p,
\end{split}
\end{align}
constitute a state-space realization as in~\eqref{eq:lqosys} of the \LQO{} error system.
The reachability and \QO{} observability Gramians $\BPe,$ $\BQe\in\R^{(n+r)\times(n+r)}$ of the error system uniquely satisfy the Lyapunov equations~\eqref{eq:reach_lyap} and~\eqref{eq:QO_obsv_lyap} for the representation in~\eqref{eq:lqosys_error_sys}, i.e.,
\begin{align}
\begin{split}
\label{eq:lqosys_err_lyap}
    \BA^{\trans} \BQe + \BQe \BAe + \BCe^{\trans}\BCe + \sum_{k=1}^p\BMe \BPe \BMe&={\Bzero}, \\
    \BAe \BPe + \BPe \BAe^{\trans} + \BBe \BBe^{\trans} &=\Bzero.
\end{split}
\end{align}
The Gramians $\BPe$ and $\BQe$ have the particular structure
\begin{align}
    \label{eq:lqosys_err_gramians}
    \BPe=\begin{bmatrix}
        \BP & \BX@@@@@\\
        @@@@@\BX^{\trans} & \BPr
    \end{bmatrix}\quad\mbox{and}\quad \BQe=\begin{bmatrix}
        \BQ & \BZ@@@@@\\
        @@@@@\BZ^{\trans} & \BQr
    \end{bmatrix},
\end{align}
where $\BP,\BQ\in\Rnn$ are the full-order system Gramians, and the submatrices $\BPr,\BQr\in\Rrr$ and $\BX,\BZ\in\Rnr$  are the unique solutions to the matrix equations
\begin{subequations}
\begin{align}
\label{eq:constraint_lyap_Pr}
\BAr\BPr + \BPr\BAr^{\trans} + \BBr\BBr^{\trans}&={\Bzero},\\
\label{eq:constraint_lyap_Qr}
\BAr^{\trans}\BQr + \BQr \BAr + \sum_{k=1}^p\BMkr\BPr\BMkr + \BCr^{\trans}\BCr&={\Bzero},\\
\label{eq:constraint_sylv_X}
\BA\BX + \BX\BAr^{\trans} + \BB\BBr^{\trans}&={\Bzero},\\
\label{eq:constraint_sylv_Z}
\BA^{\trans}\BZ + \BZ \BAr - \sum_{k=1}^p\BM_k\BX\BMkr - \BC^{\trans}\BCr&={\Bzero}.
\end{align}
\end{subequations}
Note that the matrices $\BPr$ and $\BQr$ are symmetric because they are the reachability and \QO{} observability Gramians of the reduced model $\Sysred$.
In addition to the constraints in~\eqref{eq:constraint_lyap_Pr}--\eqref{eq:constraint_sylv_Z}, it will be useful to recall the matrices $\BZ_1\in\Rnr$ and $\BQoner\in\Rrr$ in~\eqref{eq:linear_err_gramians} that satisfy
\begin{subequations}
    \begin{align}
        \label{eq:constraint_LTI_lyap_Qr}
        \BAr^{\trans}\BQoner + \BQoner\BAr +\BCr^{\trans}\BCr&=\Bzero,\\
        \label{eq:constraint_LTI_sylv_Z}
        \BA^{\trans}\BZ_1 + \BZ_1 \BAr -\BC^{\trans}\BCr&=\Bzero.
    \end{align}
\end{subequations} 
We emphasize, however, that~\eqref{eq:constraint_lyap_Qr} and~\eqref{eq:constraint_sylv_Z} are \emph{distinct} from~\eqref{eq:constraint_LTI_lyap_Qr} and~\eqref{eq:constraint_LTI_sylv_Z} when $\BM_k\neq \Bzero_{n\times n}$, with different solutions $\BQr\neq\BQoner$ and $\BZ\neq \BZ_1$.
Applying the result of~\Cref{thm:H2_from_gramians} to the realization in~\eqref{eq:lqosys_error_sys}, the squared $\CH_2$ error in \LQO-\MOR{} can be expressed as
\begin{align}
\begin{split}
\label{eq:H2err_from_err_grams}
    \CJ(\Sysred) &= \trace\left(\BBe^{\trans}\BQe\BBe\right),\\
     &= \trace\left(\BCe\BPe\BCe^{\trans}\right) + \sum_{k=1}^p\trace\left(\BPe\BMke\BPe\BMke\right).
\end{split}
\end{align}
Basic algebraic manipulations of the trace in~\eqref{eq:H2err_from_err_grams} reveal that the squared $\CH_2$ error can be re-written as
\begin{align}
\label{eq:CJ_tracechar_Q}
\CJ(\Sysred)=\trace\left(\BB^{\trans}\BQ\BB+2\BB^{\trans}\BZ\BBr+\BBr^{\trans}\BQr\BBr\right),
\end{align}
or equivalently
\begin{align}
\begin{split}
\label{eq:CJ_tracechar_P}
    \CJ(\Sysred) &=\trace\left(\BC\BP\BC^{\trans}-2\BC\BX\BCr^{\trans}+\BCr\BPr\BCr^{\trans}\right)\\
   &~~~~+\sum_{k=1}^p\trace\big(\BP\BM_k\BP\BM_k\big)\\
   &~~~~+\trace\left(\BPr\BMkr\BPr\BMkr- 2\BX^{\trans}\BM_k\BX\BMkr\right),
\end{split}
\end{align}
where $\BPr,\BQr,\BQoner\in\Rrr$ and $\BX,\BZ,\BZ_1\in\Rnr$ satisfy the constraints in~\eqref{eq:constraint_lyap_Pr}---\eqref{eq:constraint_LTI_sylv_Z}.

This brings us to our first major result, which we present in~\Cref{thm:gradients}. We leverage the formulations in~\eqref{eq:CJ_tracechar_Q} and \eqref{eq:CJ_tracechar_P} in order to establish partial gradients of the squared $\CH_2$ error $\CJ(\Sysred)=\|\Sys-\Sysred\|_{\CH_2}^2$ in \LQO-\MOR{} with respect to the reduced-order matrices $\BAr,$ $\BBr,$ $\BCr$, and $\BMkr$ for all $k=1,\ldots p$.

\begin{theorem}
    \label{thm:gradients}
    Let $\Sys$ and $\Sysred$ be asymptotically stable \LQO{} systems as in~\eqref{eq:lqosys} and~\eqref{eq:lqosys_red}, respectively. 
    Then the gradients $\gradA$, $\gradB$, $\gradC$, $\gradMk$, $k=1,\ldots, p$ of the squared $\CH_2$ error $\CJ(\Sysred)=\|\Sys-\Sysred\|_{\CH_2}^2$ are given explicitly as
    \begin{subequations}  \label{eq:gradAll}
     \begin{align}
        \label{eq:gradA}
        \gradA &= 2 \left(\left(2\BQr - \BQoner\right)\BPr +  \left(2\BZ^{\trans}-\BZ_1^{\trans}\right)\BX\right)\\
        \label{eq:gradB}
        \gradB &=2 \left(\left(2\BQr - \BQoner\right)\BBr +  \left(2\BZ^{\trans}-\BZ_1^{\trans}\right)\BB\right)\\
        \label{eq:gradC}
        \gradC &=2\left(\BCr\BPr-\BC\BX\right)\\
        \label{eq:gradM}
        \gradMk &=2\left(\BPr\BMkr\BPr-\BX^{\trans}\BM_k\BX\right), \ k = 1,\ldots, p, 
    \end{align}
  \end{subequations}
where $\BX,\BZ\in\Rnr$ and $\BPr,\BQr\in\Rrr$ satisfy the constraint equations~\eqref{eq:constraint_lyap_Pr}---\eqref{eq:constraint_sylv_Z}, and $\BZ_1\in\Rnr$, $\BQ_1\in\Rrr$ satisfy~\eqref{eq:constraint_LTI_lyap_Qr} and~\eqref{eq:constraint_LTI_sylv_Z}, respectively.
\end{theorem}

\begin{proof}
    The calculations required to prove~\Cref{thm:gradients} are conceptually intuitive, but technical. Hence, the complete proof is presented in the Appendix.
\end{proof}

\Cref{thm:gradients} now yields~\Cref{thm:foncs_gramians}, which contains our second significant theoretical contribution: The stationary points of the gradients of $\CJ$ in~\Cref{thm:gradients} characterize \FONC{}s for the $\CH_2$-optimal approximation of \LQO{} systems, and a $\CH_2$-optimal \LQO-\ROM{} is necessarily determined by Petrov-Galerkin projection~\eqref{eq:pg_proj}.
This result generalizes that of~\cite[Theorem~3.4]{VanGPA08} to the \LQO{} problem setting.

\begin{theorem}
    \label{thm:foncs_gramians}
    Let $\Sys$ and $\Sysred$ be asymptotically stable \LQO{} systems as in~\eqref{eq:lqosys} and~\eqref{eq:lqosys_red}, respectively. 
    Let $\Sysred$ be a local minimizer of the squared $\CH_2$ error $\CJ(\Sysred)=\|\Sys-\Sysred\|_{\CH_2}^2.$ Then, 
    $\Sysred$ satisfies the first-order optimality conditions
    \begin{subequations}
    \begin{align}
        \label{eq:foncs_gramians_Ar}
        \Bzero &=\left(\left(2\BQr - \BQoner\right)\BPr +  \left(2\BZ^{\trans}-\BZ_1^{\trans}\right)\BX\right)\\
        \label{eq:foncs_gramians_br}
        \Bzero &=\left(\left(2\BQr - \BQoner\right)\BBr +  \left(2\BZ^{\trans}-\BZ_1^{\trans}\right)\BB\right)\\
        \label{eq:foncs_gramians_cr}
        \Bzero &=\BCr\BPr-\BC\BX,~\mbox{and}\\
        \label{eq:foncs_gramians_Mr}
        \Bzero &=\BPr\BMkr\BPr-\BX^{\trans}\BM_k\BX, \ k=1,\ldots,p,
    \end{align}
    \end{subequations}
    where $\BX,\BZ\in\Rnr$, $\BPr,\BQr\in\Rrr$ satisfy~\eqref{eq:constraint_lyap_Pr}---\eqref{eq:constraint_sylv_Z}, and $\BZ_1\in\Rnr$, $\BQoner\in\Rrr$ satisfy~\eqref{eq:constraint_LTI_lyap_Qr} and~\eqref{eq:constraint_LTI_sylv_Z}.
    Further, if $\BPr$ and $2\BQr-\BQoner$ are nonsingular, then the $\CH_2$-optimal reduced model can be obtained via a Petrov-Galerkin projection as in~\eqref{eq:pg_proj} where the model reduction bases $\BVr,\BWr\in\Rnr$ are given by 
    \begin{equation*}
        \label{eq:opt_pg_matrices}
        \BVr=\BX\BPr^{-1}~~\mbox{and}~~\BWr={-(2\BZ-\BZ_1)(2\BQr-\BQoner)^{-1}},
    \end{equation*}
    and satisfy $\BWr^{\trans}\BVr=\BI_r$.
\end{theorem}

\begin{proof}
    The \FONC{}s in~\eqref{eq:foncs_gramians_Ar}---\eqref{eq:foncs_gramians_Mr} follow as a direct result of~\Cref{thm:gradients} along with the assumption that $\Sysred$ locally minimizes the $\CH_2$ error.
    Indeed, if $\Sysred$ is a local minimum of the function $\CJ(\Sysred)=\|\Sys-\Sysred\|_{\CH_2}^2$, then the gradients of $\CJ$ are identically zero at $\Sysred$.
    What is left to prove is that the optimal reduced model is obtained via a Petrov-Galerkin projection. 
    Because $\BPr\in\Rrr$ and $2\BQr-\BQoner\in\Rrr$ are assumed to be nonsingular, define the matrices $\BVr=\BX\BPr^{-1} \in\Rnr$ and $\BWr=-\left(2\BZ-\BZ_1\right)\left(2\BQr-\BQoner\right)^{-1}\in\Rnr$. 
    Then, the condition~\eqref{eq:foncs_gramians_Ar} implies
    \begin{align*}
        -\left(2\BQr - \BQoner\right)\BPr &= \left(2\BZ^{\trans}-\BZ_1^{\trans}\right)\BX\\
        \implies \BI_r &= -\left(2\BQr - \BQoner\right)^{-1}\left(2\BZ^{\trans}-\BZ_1^{\trans}\right)\BX \BPr^{-1}\\
        &=\BWr^{\trans}\BVr.
    \end{align*}
    Thus, $\BVr =\BX\BPr^{-1}$ and rearranging~\eqref{eq:foncs_gramians_cr} and~\eqref{eq:foncs_gramians_Mr} implies
    \begin{align*}
        \BCr&=\BC\BX\BPr^{-1}=\BC\BVr,\\
        \BMkr&=\left(\BPr^{-1}\BX^{\trans}\right)\BM_k\left(\BX\BPr^{-1}\right)=\BVr^{\trans}\BM_k\BVr,
    \end{align*}
    for each $k=1,\ldots,p$. Likewise,
    $\BWr=-(2\BZ-\BZ_1)(2\BQr-\BQoner)^{-1}$, and rearranging~\eqref{eq:foncs_gramians_br} thus implies
    \begin{align*}
            \BBr = -\left(2\BQr - \BQoner\right)^{-1}\left(2\BZ^{\trans}-\BZ_1^{\trans}\right)\BB=\BWr^{\trans}\BB.
    \end{align*}
    Lastly, $\BVr=\BX\BPr^{-1}$ implies $\BX=\BVr\BPr$. This identity and multiplying~\eqref{eq:constraint_sylv_X} by $\BWr^{\trans}$ yields
    \begin{align*}
        &\BWr^{\trans}\BA\left(\BVr\BPr\right) + \underbrace{\BWr^{\trans}(\BVr}_{=\BI_r}\BPr)\BAr^{\trans} + (\underbrace{\BWr^{\trans}\BB}_{=\BBr})\BBr^{\trans}&=\Bzero\\
        \implies ~~&\left(\BWr^{\trans}\BA\BVr\right)\BPr + \BPr\BAr^{\trans} + \BBr\BBr^{\trans}=\Bzero,
    \end{align*}
    which, by comparison with~\eqref{eq:constraint_lyap_Pr}, yields the remaining identity $\BAr=\BWr^{\trans}\BA\BVr$ since $\BPr$ is nonsingular. This completes the proof.
\end{proof}

Two remarks are in order. 

\begin{remark}
    Both the gradients of $\CJ(\Sysred)=\|\Sys-\Sysred\|_{\CH_2}^2$ in~\Cref{thm:gradients} and the Gramian-based \FONC{}s for $\CH_2$ optimality in~\Cref{thm:foncs_gramians} generalize the analogous results in the \LTI{} setting to \LQO{} systems, i.e., they establish the Wilson framework for the $\CH_2$-optimal model reduction of \LQO{} systems.
    Indeed, in the particular instance where $\BM_k=\Bzero_{n\times n}$, and so $\BMkr=\Bzero_{r\times r}$ as well, for each output $k$ and $\Sys=\Sysonlyc$ is an \LTI{} system as in~\eqref{eq:ltisys}, the reduced-order \QO{} observability Gramian $\BQr\in\Rrr$ and solution $\BZ\in\Rnr$ to~\eqref{eq:constraint_sylv_Z} reduce to $\BQr=\BQoner\in\Rrr$ in~\eqref{eq:Q1} and $\BZ=\BZ_1\in\Rnr$ solving~\eqref{eq:constraint_LTI_sylv_Z}, respectively. 
    If we apply~\Cref{thm:gradients} in this case, the gradients with respect to $\BAr$ and $\BBr$ then become
    \begin{align*}
        \gradA &= 2 \left(\BQoner\BPr +  \BZ_1^{\trans}\BX\right)\\\mbox{and} \ \
        \gradB &={2 \left(\BQoner\BBr +  \BZ_1^{\trans}\BB\right)}
    \end{align*}
    which are precisely those in~\eqref{eq:gradients_lti}; the gradient with respect to $\BCr$ is unchanged. 
    The Gramian-based \FONC{}s in~\Cref{thm:foncs_gramians} reduce in an obviously similar way. Thus,~\Cref{thm:gradients} and~\Cref{thm:foncs_gramians} contain~
    \eqref{eq:gradients_lti} and~\eqref{eq:wilson_lti} as a special case.
    \end{remark}

\begin{remark}
    In some applications, there is no linear component in the output, i.e.,  $\BC=\Bzero_{p\times n}$, so the observation term $\By=\By_2$ of~\eqref{eq:lqosys} is purely quadratic. 
    In such instances, $\BQ_1\in\Rnn$, $\BQoner\in\Rrr$ and $\BZ_1\in\Rrr$ are all zero. The gradients of $\CJ$ in~\Cref{thm:gradients} then become
    \begin{align*}
        \gradA &= 4\left(\BQr\BPr +  \BZ^{\trans}\BX\right),\\
        \gradB &=4\left(\BQr\BBr +  \BZ^{\trans}\BB\right),\\
    \gradMk &= 2\left(\BPr\BMkr\BPr-\BX^{\trans}\BM_k\BX\right), \ k =1\ldots,p,
    \end{align*} 
    where $\BQr\in\Rrr$, $\BZ\in\Rnr$ solve~\eqref{eq:constraint_lyap_Qr} and~\eqref{eq:constraint_sylv_Z} respectively, with $\BC$ and $\BCr$ equal to zero.
\end{remark}

\section{A two-sided iteration algorithm for LQO model reduction}
\label{sec:algorithm}

\Cref{thm:foncs_gramians} states that any local minimizer~\eqref{eq:lqosys_red} of the $\CH_2$ error in~\eqref{eq:H2opt_lqomor} is necessarily defined by a Petrov-Galerkin framework~\eqref{eq:pg_proj} where the optimal reduction matrices are given by $\BVr=\BX\BPr^{-1}$ and $\BWr=-\left(2\BZ-\BZ_1\right)\left(2\BQr-\BQoner\right)^{-1}$. 
It follows directly that the $\CH_2$ optimal \LQO-\ROM{} in question has an equivalent state-space realization given by
\begin{align*}
    \BAr&= \left(\left(2\BZ^{\trans}-\BZ_1^{\trans}\right)\BX\right)^{-1} \left(2\BZ^{\trans}-\BZ_1^{\trans}\right) \BA \BX\\
    \BBr&= \left(\left(2\BZ^{\trans}-\BZ_1^{\trans}\right)\BX\right)^{-1} \left(2\BZ^{\trans}-\BZ_1^{\trans}\right) \BB\\
    \BCr &= \BC\BX\\
    \BMkr&=\BX^{\trans}\BM_k\BX, \ \ k = 1\ldots,p,
\end{align*}
under the change of coordinate transformation $\BT=\BPr\in\Rrr$, with $\BT^{-1}=-\left(2\BQr-\BQoner\right)\in\Rrr$. 
In other words, an $\CH_2$-optimal \LQO-\ROM{} is defined by the projection framework~\eqref{eq:pg_proj} using the matrices $\BVr=\BX$ and $\BWr=2\BZ-\BZ_1$,
where $\BWr^{\trans}\BVr=\left(2\BZ^{\trans}-\BZ_1^{\trans}\right)\BX$ is invertible due to the identity from~\eqref{eq:foncs_gramians_Ar}, namely
\begin{equation*}
    \left(2\BZ^{\trans}-\BZ_1^{\trans}\right)\BX=-\left(2\BQr-\BQoner\right)\BPr,
\end{equation*}
and the assumption that $2\BQr-\BQoner$ and $\BPr$ are nonsingular. 
The right projection matrix $\BVr=\BX$ satisfies the Sylvester equation~\eqref{eq:constraint_sylv_X} corresponding to the $\CH_2$-optimal \LQO-\ROM. Because $\BZ\in\Rnr$ and $\BZ_1\in\Rnr$ are the solutions to~\eqref{eq:constraint_sylv_Z} and~\eqref{eq:constraint_LTI_sylv_Z}, the left projection matrix $\BWr =\wh\BZ\coloneqq2\BZ-\BZ_1$ satisfies a linear combination of~\eqref{eq:constraint_sylv_Z} and~\eqref{eq:constraint_LTI_sylv_Z}, i.e.,
\begin{align}
\begin{split}       \label{eq:opt_sylv_eqs}
    \BA^{\trans}\wh\BZ+ \wh\BZ \BAr - 2\sum_{k=1}^p\BM_k\BX\BMkr - \BC^{\trans}\BCr&=\Bzero\\
\mbox{where}~\BX\in\Rnr~\mbox{satisfies}~~\BA\BX+\BX\BAr^{\trans}+\BB\BBr^{\trans}&=\Bzero.
\end{split}
\end{align}
Note that the Sylvester equations in~\eqref{eq:opt_sylv_eqs} depend explicitly upon the $\CH_2$ optimal reduced model. 
Thus, if we wanted to project down the full-order matrices using the optimal choice of $\BVr=\BX$ and $\BWr=\wh\BZ$, this would require \emph{a priori} knowledge of the optimal reduced-order matrices $\BAr$, $\BBr$, $\BCr$, and $\BMkr$, which is, of course, unavailable. 

Based on these observations, we propose a computationally efficient procedure for $\CH_2$-optimal \LQO-\MOR{} that performs iterated projection with the solutions to the Sylvester equations~\eqref{eq:opt_sylv_eqs} in~\Cref{alg:lqo_tsia}.
This leads to a fixed point iteration where at every step, the equations~\eqref{eq:opt_sylv_eqs} are solved to project the full-order \LQO{} system and obtain a new reduced model. This procedure is repeated until some stopping criterion is satisfied. 
The idea of using fixed-point algorithms in $\CH_2$-optimal model reduction is not new; indeed, ours is inspired by \TSIA{}~\cite{XuZ11}, which was mentioned briefly in~\Cref{ss:h2_opt_lti}. \TSIA{} performs iterative projection on the \LTI{} full-order model using the solutions to~\eqref{eq:constraint_sylv_X} and~\eqref{eq:constraint_LTI_sylv_Z}.
Because of this, we call~\Cref{alg:lqo_tsia} the \emph{linear quadratic output two-sided iteration algorithm} $(\LQO\mbox{-}\TSIA)$.

\begin{algorithm}[h!]
\caption{Linear quadratic output two-sided iteration algorithm $(\LQO\mbox{-}\TSIA)$ for $\CH_2$-optimal \LQO-\MOR{}}
\begin{algorithmic}[1]
    \label{alg:lqo_tsia}
    \REQUIRE order-$n$ \LQO-\FOM{} $\Sys=\left(\BA,\BB,\BC,\BM_1,\ldots,\BM_k\right)$, $1\leq r<n$, and initial order-$r$ \LQO-\ROM{} $\Sysred^{(0)}=\big(\BAr^{(0)},\BBr^{(0)},\BCr^{(0)},\BMoner^{(0)},\ldots,\BMpr^{(0)}\big)$ with $\lambda(\BA)$, $\lambda(\BAr^{(0)})\subset\C_-$, $\BM_k,$ $\BMkr^{(0)}$ symmetric for all $k$, and convergence tolerance $\epsilon>0$.
    \vspace{2mm}
    \STATE \textbf{while error} $>\epsilon$ \textbf{do}
    \vspace{2mm}
     \STATE Solve the Sylvester equations~\eqref{eq:opt_sylv_eqs} for $\BX,\wh\BZ \in \Rnr$:
    \begin{align*} 
        \BA\BX + {\BX\BAr^{(j)}}^{\trans} + \BB{\BBr^{(j)}}^{\trans}&=\Bzero,\\
        \BA^{\trans}\wh\BZ + \wh\BZ \BAr^{(j)} - 2\sum_{k=1}^p\BM_k\BX\BMkr^{(j)} - \BC^{\trans}\BCr^{(j)}&=\Bzero.
    \end{align*}
    \vspace{-2mm}
    \STATE Perform orthogonalization on the solution matrices:
    \begin{align*} 
        \BVr = \text{orth}(\BX), \ \ \BWr = \text{orth}(\wh\BZ).
    \end{align*} 
    \vspace{-4mm}
    \STATE Compute reduced-order matrices:
    \begin{align*}
        \BAr^{(j+1)} &= (\BWr^{\trans}\BVr)^{-1} \BWr^{\trans} \BA \BVr,\\
        \BBr^{(j+1)} &= (\BWr^{\trans}\BVr)^{-1} \BWr^{\trans}\BB, \\
        \BCr^{(j+1)} &=  \BC\BVr, \\
        \BMkr^{(j+1)} &= \BVr^{\trans}  \BM_k \BVr,~~k=1,\ldots,p.
    \end{align*} 
    \vspace{-4mm}
    \STATE \textbf{end while}.
    \vspace{2mm}
    \ENSURE  Order-$r$ \LQO-\ROM{} specified by the reduced order matrices $\Sysred=\left(\BAr,\BBr,\BCr,\BMoner,\ldots,\BMpr\right)$.
    \end{algorithmic} 
\end{algorithm}

\subsection{Implementation details}
\label{ss:numerical_details}
Here, we discuss specific implementation details of~\Cref{alg:lqo_tsia} such as methods for computing solutions to the Sylvester equations~\eqref{eq:opt_sylv_eqs} and various stopping criteria.

\subsubsection{Solving the Sylvester equations in~\eqref{eq:opt_sylv_eqs}}
Evidently, the main computational cost at each step of the iteration is the solution of two Sylvester equations in Step 2 of~\Cref{alg:lqo_tsia}.
Because the solution matrices $\BX,\wh\BZ\in\Rnr$ in~\eqref{eq:opt_sylv_eqs} are tall and skinny, they can be obtained efficiently and directly by computing a Schur decomposition of \emph{the reduced matrix} $\BAr$, and solving for their columns via shifted linear solves.
Direct methods for solving this type of Sylvester equation are provided in~\cite{BenKS11}. For completeness, we briefly describe how these methods can be applied to the equation for $\wh\BZ$ in~\eqref{eq:opt_sylv_eqs}.

To be precise, let $\BAr^{\trans}=\BU^{\herm}\BS\BU$ be the Schur form of $\BAr^{\trans}$ so that $\BU\in\Crr$ is unitary, and $\BS\in\Crr$ is an upper triangular matrix.
Replacing $\BAr^{\trans}$ with its Schur form in the equation for $\wh\BZ$ in~\eqref{eq:opt_sylv_eqs} and multiplying on the right by $\BU^{\herm}$ produces a similar Sylvester equation
\begin{align}
\begin{split}
   \label{eq:triang_sylv_Z}
    &\BA^{\trans}\left(\wh{\BZ}\BU^{\herm}\right)+\left(\wh{\BZ}\BU^{\herm}\right)\BS^{\trans}\\
    &~~~- 2\sum_{k=1}^p \BM_k\BX\BMkr\BU^{\herm}-\BC^{\trans}\BCr\BU^{\herm}=\Bzero. 
\end{split}
\end{align}
The columns of $\wh{\BZ}\BU^{\herm}$ are computed by backward substitution;
the matrix $\BX$ can be computed using a nearly identical procedure.
Because $\BAr$ is a small $r\times r$ matrix, its Schur form is computable in $O(r^3)$ operations where $r\ll n$. Thus, the dominant cost in obtaining $\BX$ and $\wh{\BZ}$ lies in solving the $2r$ shifted linear systems of equations that result from the transformed equations~\eqref{eq:opt_sylv_eqs}.
Usually, the large-scale coefficient matrix $\BA$ has some inherent sparsity that can be taken advantage of. So, the complexity of the algorithm is roughly $2r\,O(S)$, where $O(S)$ is the complexity of the solver used.
In the worst case, where $\BA$ is dense, $O(S)$ is bounded above by $O(n^3)$. However, modern solvers are typically much more efficient, and so the complexity $2r\,O(S)$ will likely be more favorable than $O(n^3)$; we refer the reader to~\cite[Sec.~3]{BenKS11} and the references therein. Lastly, while the described procedure involves complex arithmetic, this can be avoided by using the real-valued block Schur form of $\BAr$ instead.

Significantly, this shows that the Sylvester equations in~\eqref{eq:opt_sylv_eqs} can be solved \emph{directly} using only \emph{sparse} calculations involving the full-order coefficient matrix $\BA$. As a point of comparison, the structure-preserving balanced truncation $(\BTr{})$ from~\cite{BenGPD21}, which we refer to as \LQO-\BTr{}, requires the one-time solution of the two large-scale Lyapunov equation~\eqref{eq:reach_lyap} and~\eqref{eq:QO_obsv_lyap} to obtain the \LQO{} system Gramians.  
Solving these via direct methods, such as the Bartels-Stewart algorithm, requires 
the Schur decomposition of the large-scale $\BA$ matrix and has a complexity of $O(n^3)$.

\subsubsection{Convergence criteria and stability of the \textsf{ROM}}
\label{sss:conv}
The iteration in~\Cref{alg:lqo_tsia} repeats until either some preset number of steps is reached, or the algorithm converges within the tolerance $\epsilon>0$ based on some pre-determined stopping criterion.
As is the case with any optimization problem, there exist a variety of possible choices for measuring convergence.
Because we are seeking to minimize the squared $\CH_2$ error $\CJ$ in~\eqref{eq:H2opt_lqomor}, for simplicity, we use the change in the (relative) squared $\CH_2$ error between consecutive iterates to monitor convergence in~\Cref{alg:lqo_tsia}.
From~\eqref{eq:CJ_tracechar_Q}, the square of the relative error due to $\Sysred^{(j)}$ at step $j$ of the iteration is 
\begin{align}
\label{eq:err_conv}
{\eta^{(j)}}\coloneqq
 \frac{\|\Sys-\Sysred^{(j)}\|_{\CH_2}^2}{\|\Sys\|_{\CH_2}^2}=\frac{\|\Sys\|_{\CH_2}^2+\|\Sysred^{(j)}\|_{\CH_2}^2+2\trace\big(\BB\BZ{\BBr^{(j)}}^{\trans}\big)}{\|\Sys\|_{\CH_2}^2},
\end{align}
where $\BZ\in\Rnr$ satisfies~\eqref{eq:constraint_sylv_Z} for $\Sysred^{(j)}$. Then,~\Cref{alg:lqo_tsia} is deemed to have converged if 
\begin{equation}
    \label{eq:H2errConv}
    |\eta^{(j)}-\eta^{(j-1)}|/\eta^{(1)}\leq \epsilon~~\mbox{for}~~\epsilon>0,
\end{equation}
i.e., when the normalized change in $\eta^{(j)}$ is $\leq\epsilon$.
The $\CH_2$ norm of the \FOM{} in~\eqref{eq:err_conv} can be precomputed before the iteration.

Another natural option is to monitor changes in the gradients $\gradA$, $\gradB$, $\gradC$, and $\gradMk$ of the error function $\CJ$, and terminate when they are sufficiently small. 
The information required to compute these quantities is readily available from the iteration itself. However, a relative metric that uses scaled gradients would require computing the Hessian of $\CJ$, which is not directly available from already computed quantities. So, we do not consider this criterion further.

Upon convergence of~\Cref{alg:lqo_tsia}, the optimality conditions in~\Cref{thm:foncs_gramians} are satisfied.
As is the case for the linear \TSIA~\cite{XuZ11} (and~\IRKA~\cite{GugAB08} for the interpolatory formulation of the linear $\CH_2$-optimality framework), convergence of~\Cref{alg:lqo_tsia} is \emph{not} guaranteed in general because it is a fixed-point iteration. However, similar to \TSIA{} and \IRKA{}, in practice, the algorithm performs well. 
We include a study of the convergence of~\Cref{alg:lqo_tsia} in~\Cref{sec:numerics}. 
For guaranteed convergence, one may consider developing a descent-based algorithm based on the explicit gradient formulae~\eqref{eq:gradAll} we derived.
We leave such considerations to future work.
\begin{remark}
    \label{remark:conv}
    Computing the $\CH_2$ norm of the \FOM{} entails the solution of the large-scale Lyapunov equation in~\eqref{eq:QO_obsv_lyap}, which, as already discussed, may not be feasible for truly large-scale problems.
    The examples we consider in~\Cref{sec:numerics} have a modest state-space dimension, and so the $\CH_2$ norm of the full-order system $\Sys$ can be computed without issue.
    However, note that in the error formula~\eqref{eq:err_conv}, and thus the criterion~\eqref{eq:H2errConv}, the only part that varies in each iteration is the `tail', i.e.,
    \begin{equation}
        \label{eq:err_tails}
        \tau^{(j)}\coloneqq \|\Sysred^{(j)}\|_{\CH_2}^2+2\trace\big(\BB\BZ{\BBr^{(j)}}^{\trans}\big).
    \end{equation}
  Therefore, if computing the true (or an approximate) $\CH_2$ norm of the \FOM{} is not feasible, one may instead monitor the relative change in the tails~\eqref{eq:err_tails}. The corresponding convergence criterion is
  \begin{equation}
    \label{eq:H2tailConv}
    |\tau^{(j)}-\tau^{(j-1)}|/|\tau^{(1)}|\leq \epsilon~~\mbox{for}~~\epsilon>0.
\end{equation}
It is straightforward to verify that if the quantity~\eqref{eq:H2errConv} \textit{is} decreasing throughout the iteration, i.e., $\eta^{(j+1)}-\eta^{(j)} \leq \eta^{(j)}-\eta^{(j-1)}$, then it follows that $\tau^{(j+1)}-\tau^{(j)} \leq \tau^{(j)}-\tau^{(j-1)}$, and vice versa.
We investigate the behavior of the proposed convergence criteria~\eqref{eq:H2errConv} and~\eqref{eq:H2tailConv} in~\Cref{sec:numerics}.
\end{remark}
As in the purely \LTI{} case, there is no guarantee that the intermediate or converged \LQO-\ROM{} will be asymptotically stable. 
In our numerical examples, intermediate unstable \ROM{}s are rarely encountered during the iteration, and we have never observed convergence to an unstable \ROM{}.

We conclude this section with a comment regarding the recent work~\cite[Algorithm~1]{GosA19}
where a two-sided iteration similar to~\Cref{alg:lqo_tsia} was proposed. The algorithm in~\cite{GosA19} was proposed without making explicit reference to $\CH_2$ optimality conditions of~\Cref{thm:foncs_gramians}; it was developed heuristically based on the corresponding two-sided iteration for linear systems in~\cite{XuZ11}.  As a result, despite the structural similarities, the method of~\cite{GosA19} solves a different Sylvester equation to compute $\wh\BZ$ and thus, unlike 
~\Cref{alg:lqo_tsia}, the resulting \LQO-\ROM{} does not satisfy the $\CH_2$ optimality conditions in~\Cref{thm:foncs_gramians} upon convergence.

\section{Numerical Results}
\label{sec:numerics}
In this section, we test the effectiveness of the approach in \Cref{alg:lqo_tsia} on a model problem.
All experiments were performed on a MacBook Air with 8 gigabytes of RAM and an Apple M2 processor running macOS Ventura version 13.4 with MATLAB 23.2.0.2515942 (R2023b) Update 7.
The source codes for recreating the numerical experiments and the computed results are
available at~\cite{supRei24}.

\subsection{Model problem and experimental setup}
\label{ss:advecdiff_example}
For testing the proposed \MOR{} approach in Algorithm~\ref{alg:lqo_tsia}, we use the example of a 1D Advection-diffusion equation from~\cite[Section~4.1]{DiazHGA23}.
The governing equations are 
\begin{align}
\begin{split}
\label{eq:advec}
    \frac{\partial}{\partial t} v(t,x)-\alpha \frac{\partial^2}{\partial x^2} v(t,x) + \beta \frac{\partial}{\partial x} v(t,x)&=0,\\
    v(t,0)=u_0(t), \ \ \alpha \frac{\partial}{\partial x} v(t,1)=u_1(t), \ \ v(0,x)&=0,
\end{split}
\end{align}
for $x\in(0,1)$ and $t\in(0,T)$ and inputs $u_0,u_1\in\CL_2(0,T)$; the diffusion and advection coefficients are $\alpha>0$ and $\beta \geq0$, respectively.
The output that we consider is
\begin{align}
\label{eq:quadcost}
    \frac{1}{2}\int_0^1|v(t,x)-1|^2@d@x,
\end{align}
Such an observable may arise from, e.g., the quadratic cost function in an optimal control problem.
Discretizing the equations in~\eqref{eq:advec} using $n+1$ equidistant spatial points yields an order-$n$ state-space model of the form~\eqref{eq:lqosys} with $m=2$ inputs ($u_0$ and $u_1$) and $p=1$ output $y$. Let $\Bx(t)\in\Rn$ denote the spatial discretization of $v(t,x)$,  $h\coloneqq 1/n$, and $\mathbf{1}\in\Rn$ the vector consisting of all ones.
Then, the discretization provides an approximation to the quadratic cost function~\eqref{eq:quadcost} 
\begin{equation*}
    \frac{h}{2}\|\Bx(t)-\mathbf{1}\|_2^2=\underbrace{-h@\mathbf{1}^{\trans}\Bx(t)}_{=y_1(t)} + \underbrace{\frac{h}{2}\Bx(t)^{\trans}\Bx(t)}_{=y_2(t)} + \frac{h}{2}\|\mathbf{1}\|_2^2.
\end{equation*}
To fit~\eqref{eq:lqosys}, the single output of the discretized system is given by $y(t)=\Bc\Bx(t) + \Bx(t)^{\trans}\BM\Bx(t)$ for $\Bc\coloneqq-h@\mathbf{1}^{\trans}\in\R^{1\times n}$ and $\BM\coloneqq\frac{h}{2}\BI_n\in\Rnn$, where $\BI_n$ is the $n\times n$ identity matrix. The approximation to the cost~\eqref{eq:quadcost} is recovered from $y(t)$ via $\frac{h}{2}\|\Bx(t)-\mathbf{1}\|_2^2=y(t) + \frac{h}{2}\|\mathbf{1}\|_2^2.$
An upwind finite-difference discretization of~\eqref{eq:advec} is performed using $n+1=3001$ spatial grid points to obtain an \LQO{} system in state-space form~\eqref{eq:lqosys}; the diffusion and advection parameters are selected as $\alpha=1$ and $\beta=1$, respectively.

For the model reduction of the \LQO{} system described above, we test the proposed \LQO-\TSIA{} approach in~\Cref{alg:lqo_tsia} against the \LQO-\BTr{} approach discussed in Section~\ref{ss:numerical_details} from~\cite{BenGPD21}.
To test the robustness of \LQO-\TSIA{} for different starting values, we compare three different strategies for selecting the initial \ROM{} parameters $\BAr^{(0)}\in\Rrr$, $\BBr^{(0)}\in\Rrm$, $\BCr^{(0)}\in\Rpr$ and $\BMr^{(0)}\in\Rrr$:
\begin{description}
    \item[\LQOTSIAdiag{}] uses a diagonal matrix of $r$ logarithmically-spaced points in the interval from $-10^{0}$ and $-10^{4}$ for $\BAr^{(0)}$, the leading $m$ columns and $p$ rows of the identity matrix for $\BBr^{(0)}$ and $\BCr^{(0)}$, and $\BMr^{(0)}=\BI_{r\times r}$; 
    \item[\LQOTSIAtrunc{}] uses the leading (truncated) $r$-dimensional blocks of the full-order model matrices;
    \item[\LQOTSIAeigs{}] uses an initial \ROM{} computed by Galerkin projection, where $\BWr=\BVr$ in~\eqref{eq:pg_proj} are taken to be an orthonormalized subset of $r$ eigenvectors of $\BA$, computed using MATLAB's \textsf{eigs} command. 
\end{description} 
For computing the solutions to the Sylvester equations in~\eqref{eq:opt_sylv_eqs} during Step 2 of~\Cref{alg:lqo_tsia}, we use the function \\\texttt{mess\textunderscore{}sylvester\textunderscore{}sparse\textunderscore{}dense} from  the M-M.E.S.S. library~\cite{SaaKB21-mmess-3.0} version 3.0.
For computing solutions to the Lyapunov equations~\eqref{eq:reach_lyap} and~\eqref{eq:obsv_lyap} in \LQO-\BTr, MATLAB's \texttt{lyap} command is used.
For each initialization strategy, \LQO-\TSIA{} is run with an overly strict tolerance of $\epsilon=10^{-14}$; in practice, one would choose a larger magnitude stopping criterion. 
We use the normalized change in the tails of the squared $\CH_2$ error~\eqref{eq:H2tailConv} to determine convergence; we record~\eqref{eq:H2errConv} as well to compare how these criteria evolve during the iteration.

We test the performance of the computed \LQO-\ROM{}s in approximating full-order time-domain output trajectories $y$ for particular choices of inputs $u_1$ and $u_2$.
The following error measures are used to assess the approximation quality of the reduced outputs:
To visibly compare the performance of the reduced models, we plot the full- and reduced-order outputs, as well as their pointwise relative error as given by
\begin{equation}
    \label{eq:pointwiseOutputError}
    \relerr(t_i)\coloneqq\frac{|y(t_i)-\yr(t_i)|}{|y(t_i)|},~~t_i\in[t_{{\min}},t_{{\max}}],
\end{equation}
where $t_i\in[t_{\min},t_{\max}]$ are the $N$ timesteps in the simulation.
To assess the average and worst-case performance of the \ROM{}s over the simulation, we use approximations of the relative $\CL_2$ error and $\CL_\infty$ error
\begin{align}
    \label{eq:relL2error}
    \relerr_{\CL_2}&\coloneqq\left(\frac{\sum_{i=1}^N |y(t_i)-\yr(t_i)|^2}{\sum_{i=1}^N|y(t_i)|^2}\right)^{1/2}\\
    \label{eq:relLinftyerror}
    \relerr_{\CL_\infty}&\coloneqq\max_{t_i\in[t_{{\min}},t_{{\max}}]}\frac{|y(t_i)-\yr(t_i)|}{|y(t_i)|}.
\end{align}
We also score the reduced model performance using the squared relative $\CH_2$ error given according to~\Cref{def:H2norm}.
\begin{equation}
    \label{eq:relH2error}
    \relerr_{\CH_2}\coloneqq\frac{\|\Sys-\Sysred\|_{\CH_2}^2}{\|\Sys\|_{\CH_2}^2}.
\end{equation}

\begin{figure*}[t!]
    \centering
    \begin{subfigure}[t!]{.49\linewidth}
        \raggedleft
        \begin{tikzpicture}[font = \plotfontsize]
  \pgfplotstableread{graphics/data/AdvecDiff3000_sinusoidal_r30_Outputs.dat}\tableINPUT
  
  \begin{axis}[%
    width  = .8\linewidth,
    height = .1\textheight,
    scale only axis,
    xmin = 0,
    xmax = 10,
    ymin = -.55,
    ymax = .05,
    xminorticks = true,
    yminorticks = true,
    xlabel = {time $t$ (s)},
    ylabel = {outputs},
    ylabel style   = {yshift = -.3em},
    scaled x ticks = false,
    x tick label style = {/pgf/number format/1000 sep={\,}},
    y tick label style = {/pgf/number format/1000 sep={\,}},
    cycle list name    = plotlist
  ]
  
    \foreach \y in {1, 2, ..., 5}{
      \addplot+ table[x index = 0, y index = \y] {\tableINPUT};
    }
  \end{axis}
\end{tikzpicture}\\
        \begin{tikzpicture}[font = \plotfontsize]
    \pgfplotstableread{graphics/data/AdvecDiff3000_sinusoidal_r30_OutputErrors.dat}\tableINPUT
  
  \begin{semilogyaxis}[%
    width  = .8\linewidth,
    height = .1\textheight,
    scale only axis,
    xmin = 0,
    xmax = 10,
    ymin = 0,
    ymax = 1e-3,
    xminorticks = true,
    yminorticks = true,
    xlabel = {time $t$ (s)},
    ylabel = {$\relerr(t)$},
    ylabel style   = {yshift = -.3em},
    scaled x ticks = false,
    x tick label style = {/pgf/number format/1000 sep={\,}},
    y tick label style = {/pgf/number format/1000 sep={\,}},
    cycle list name    = plotlist
  ]

  \pgfplotsset{cycle list shift = 1}
  
    \foreach \y in {1, 2, ..., 4}{
      \addplot+ table[x index = 0, y index = \y] {\tableINPUT};
    }
  \end{semilogyaxis}
\end{tikzpicture}\\
    \caption{Output magnitudes and pointwise relative errors~\eqref{eq:pointwiseOutputError} for inputs $u_0(t)=0$ and $u_{\cos}(t)=.5\cos(\pi t)+1$.}
    \label{fig:r30advecdiff_output_sin}
    \vspace{.5\baselineskip}
    \end{subfigure}%
    \hfill%
    \begin{subfigure}[t!]{.49\linewidth}
        \raggedleft
        \begin{tikzpicture}[font = \plotfontsize]
  \pgfplotstableread{graphics/data/AdvecDiff3000_exponential_r30_Outputs.dat}\tableINPUT
  
  \begin{axis}[%
    width  = .8\linewidth,
    height = .1\textheight,
    scale only axis,
    xmin = 0,
    xmax = 30,
    ymin = 0,
    ymax = 90,
    xminorticks = true,
    yminorticks = true,
    xlabel = {time $t$ (s)},
    ylabel = {outputs},
    ylabel style   = {yshift = -.3em},
    scaled x ticks = false,
    x tick label style = {/pgf/number format/1000 sep={\,}},
    y tick label style = {/pgf/number format/1000 sep={\,}},
    cycle list name    = plotlist
  ]
  
    \foreach \y in {1, 2, ..., 5}{
      \addplot+ table[x index = 0, y index = \y] {\tableINPUT};
    }
  \end{axis}
\end{tikzpicture}\\
        \begin{tikzpicture}[font = \plotfontsize]
    \pgfplotstableread{graphics/data/AdvecDiff3000_exponential_r30_OutputErrors.dat}\tableINPUT
  
  \begin{semilogyaxis}[%
    width  = .8\linewidth,
    height = .1\textheight,
    scale only axis,
    xmin = 0,
    xmax = 30,
    ymin = 0,
    ymax = 1e-2,
    xminorticks = true,
    yminorticks = true,
    xlabel = {time $t$ (s)},
    ylabel = {$\relerr(t)$},
    ylabel style   = {yshift = -.3em},
    scaled x ticks = false,
    x tick label style = {/pgf/number format/1000 sep={\,}},
    y tick label style = {/pgf/number format/1000 sep={\,}},
    cycle list name    = plotlist
  ]

  \pgfplotsset{cycle list shift = 1}
  
    \foreach \y in {1, 2, ..., 4}{
      \addplot+ table[x index = 0, y index = \y] {\tableINPUT};
    }
  \end{semilogyaxis}
\end{tikzpicture}\\
    \caption{Output magnitudes and pointwise relative errors~\eqref{eq:pointwiseOutputError} for inputs $u_0(t)=0$ and $u_{\exp}(t)=t^2e^{-t/2}$.}
    \label{fig:r30advecdiff_output_exp}
    \vspace{.5\baselineskip}
    \end{subfigure}%

  \begin{tikzpicture}
  \begin{axis}[%
    hide axis,
    width  = 1mm,
    height = 1mm,
    scale only axis,
    xmin = 0,
    xmax = 1,
    ymin = 0,
    ymax = 1,
    legend columns = 5, 
    legend style   = {
      at     = {(0,0)},
      anchor = center,
      /tikz/every even column/.append style = {column sep = 0.2cm}},
    legend cell align  = {left},
    clip mode          = individual,
    cycle list name    = plotlist]

    \foreach \y in {1, 2, ..., 5}{
      \addplot+ coordinates{ (0, 0) };
    }
    \addlegendentry{\ensuremath{\mathsf{FOM}}}
    \addlegendentry{\LQOTSIAdiag{}}
    \addlegendentry{\LQOTSIAtrunc{}}
    \addlegendentry{\LQOTSIAeigs{}}
    \addlegendentry{\LQOBT{}}
  \end{axis}
\end{tikzpicture}
  \vspace{.5\baselineskip}
    \label{fig:outputErrors}
  \caption{Output magnitudes and pointwise relative errors~\eqref{eq:pointwiseOutputError} of the full-order and order $r=30$ \LQO-\ROM{}s driven by $u_{\cos}$ and $u_{\exp}$. Each of the \LQO-\TSIA{} reduced models and the \LQO-\BTr{} reduced model provide reasonably accurate recreations of the \FOM{} output, although the approximations produced by \LQO-\TSIA{} perform slightly better overall.}
\end{figure*}

\begin{table*}[t!]
  \centering
  \begin{tabular}{lllll}
    \hline\noalign{\smallskip}
      & \multicolumn{1}{c}{\LQOTSIAdiag{}}
      & \multicolumn{1}{c}{\LQOTSIAtrunc{}}
      & \multicolumn{1}{c}{\LQOTSIAeigs{}}
      & \multicolumn{1}{c}{\LQOBT{}}\\
      \noalign{\smallskip}\hline\noalign{\smallskip}
      $\relerr_{\CL_2}$~ ($u_{\cos}$) & $1.1079\texttt{e-}6$
      & $1.1499\texttt{e-}6$
      & $\boldsymbol{1.0454\texttt{e-}6}$
      & $5.5315\texttt{e-}5$\\
        $\relerr_{\CL_2}$~ ($u_{\exp}$) & $4.5888\texttt{e-}6$
      & $4.7583\texttt{e-}6$
      & $\boldsymbol{4.4079\texttt{e-}6}$
      & $1.2086\texttt{e-}5$\\
      $\relerr_{\CL_\infty}$ ($u_{\cos}$) & $6.0951\texttt{e-}6$
      & $6.5111\texttt{e-}6$
      & $\boldsymbol{5.6404\texttt{e-}6}$
      & $6.0878\texttt{e-}4$\\
        $\relerr_{\CL_\infty}$ ($u_{\exp}$) & $4.7676\texttt{e-}5$
      & $4.7680\texttt{e-}5$
      & $\boldsymbol{4.7674\texttt{e-}5}$
      & $2.3879\texttt{e-}2$\\
      \noalign{\smallskip}\hline\noalign{\smallskip}
  \end{tabular}
  \caption{Relative $\CL_{2}$, $\CL_\infty$ errors~\eqref{eq:relL2error},~\eqref{eq:relLinftyerror} for the order $r=30$ \LQO-\ROM{}s. The smallest error is highlighted in \textbf{boldface}.}
  \label{tab:relErrors}
\end{table*} 

\subsection{Discussion of results}

We compute \LQO-\ROM{}s of order $r=30$ using the \LQOTSIAdiag{}, \LQOTSIAtrunc{}, \LQOTSIAeigs{}, and \LQOBT{} approaches described in~\Cref{ss:advecdiff_example}.
We perform time-domain simulations of the full- and reduced-order outputs using two different pairs of input signals; in either case, we enforce a boundary condition of $u_0(t)=0.$ 

We first simulate the full- and reduced-order models by applying a sinusoidal input  $u_1(t)=u_{\cos}(t)\coloneqq.5\cos(\pi t)+1$, and secondly by applying an exponentially damping quadratic input $u_1(t)=u_{\exp}(t)\coloneqq t^2 e^{-t/5}$. 
To capture the full spread of the output dynamics in each experiment, the first simulation occurs over $T=10$ seconds, and the second over $T=30$ seconds.
We plot the output trajectories of the full and reduced-order systems due to $u_{\cos}(t)$ and $u_{\exp}(t)$ and the associated relative pointwise error~\eqref{eq:pointwiseOutputError} in~\Cref{fig:r30advecdiff_output_sin} and~\Cref{fig:r30advecdiff_output_exp}, respectively.
The relative $\CL_2$ and $\CL_\infty$ output errors according to~\eqref{eq:relL2error} and~\eqref{eq:relLinftyerror} are reported in~\Cref{tab:relErrors}.
As evidenced by the plots of the pointwise relative output error~\eqref{eq:pointwiseOutputError}, both the \LQO-\TSIA{} and \LQO-\BTr{} \ROM{}s are successful in replicating the full-order output $y$ for both pairs of input signals.
For the sinusoidal input $u_{\cos}$,~\Cref{fig:r30advecdiff_output_sin} as well as~\Cref{tab:relErrors} show that the \LQO-\TSIA{} reduced models perform a couple of orders of magnitude better than \LQO-\BTr{} on average and in the worst case. 
For the exponential input $u_{\exp}$, both methods perform similarly well on average as seen in Figure~\ref{fig:r30advecdiff_output_exp}. The \LQO-\TSIA{} \ROM{}s perform slightly better in the $\CL_\infty$ metric, given that the \LQOBT{} \ROM{} produces a large error near $t=0$.
We note that all three of the \LQO-\TSIA{} \ROM{}s produce very similar approximations, despite the different initialization strategies.

The timings and iteration counts (if applicable) for computing the various \LQO-\ROM{}s are reported in~\Cref{tab:runtimes}.
Even with a modest model order of $n=3000$, we observe that \LQO-\TSIA{} produces a \LQO-\ROM{} faster than \LQO-\BTr{}, given that the Lyapunov equations in~\eqref{eq:reach_lyap} and~\eqref{eq:obsv_lyap} do not need to be solved. We believe this scaling proves promising, and that \LQO-\TSIA{} may be successful when applied to truly large-scale problems.

\begin{table}[t!]
  \centering
  \resizebox{\linewidth}{!}{
  \begin{tabular}{lllll}
    \hline\noalign{\smallskip}
      & \multicolumn{1}{c}{\LQOTSIAdiag{}}
      & \multicolumn{1}{c}{\LQOTSIAtrunc{}}
      & \multicolumn{1}{c}{\LQOTSIAeigs{}}
      & \multicolumn{1}{c}{\LQOBT{}}\\
      \noalign{\smallskip}\hline\noalign{\smallskip}
      Run time (s) & $7.94$ s 
      & $4.61$ s
      & $3.84$ s 
      & $29.12$ s\\
     Iteration count & $48$ & $28$ & $29$ & N/A\\
      \noalign{\smallskip}\hline\noalign{\smallskip}
  \end{tabular}
  }%
  \caption{Run times and iteration counts for computing the order $r=30$ \LQO-\ROM{}s via \LQOTSIAdiag{}, \LQOTSIAtrunc{}, \LQOTSIAeigs{}, and \LQOBT{}. The convergence criteria used for the \LQO-\TSIA{} \ROM{}s is the normalized change in the tails of the squared $\CH_2$ error~\eqref{eq:H2tailConv} with a tolerance of $\epsilon=10^{-14}$.}
  \label{tab:runtimes}
\end{table} 

To understand the behavior of the two proposed convergence criteria~\eqref{eq:H2errConv} and~\eqref{eq:H2tailConv} discussed in~\Cref{sss:conv}, we run each of \LQOTSIAdiag{}, \LQOTSIAtrunc{}, and \LQOTSIAeigs{} \emph{twice} for $r = 30$. First, the normalized change in the tails of the squared $\CH_2$ error~\eqref{eq:H2tailConv} is used to determine convergence of \LQOTSIAdiag{}, \LQOTSIAtrunc{}, and \LQOTSIAeigs{}; second, the normalized change in the $\CH_2$ errors~\eqref{eq:H2errConv} is used to determine convergence.
The magnitude of either~\eqref{eq:H2errConv} or~\eqref{eq:H2tailConv} throughout each of the six iterations is plotted in~\Cref{fig:advecdiff_conv}. 
The different initialization strategies exhibit very similar convergence behavior overall. 
Both \LQOTSIAdiag{} and\\\LQOTSIAeigs{} find a local minimum and converge after less than $50$ iterations using either convergence criteria. For \LQOTSIAtrunc{}, the algorithm converges after 28 iterations using~\eqref{eq:H2tailConv}. When using~\eqref{eq:H2errConv}, the error stagnates around $10^{-12}$ once the change in tails begins to hover around the order of machine precision.
We observe that the normalized change in the tails~\eqref{eq:H2tailConv} tracks very well with the normalized change in the $\CH_2$ errors~\eqref{eq:H2errConv} throughout each comparable iteration, and conclude that this is a reliable alternative for monitoring convergence if computing the $\CH_2$ norm of the \FOM{} is intractable.

\begin{figure}[t!]
\centering
\raggedleft
\begin{tikzpicture}[font = \plotfontsize]

\begin{semilogyaxis}[%
    width  = .8\linewidth,
    height = .15\textheight,
    scale only axis,
    xmin = 1,
    xmax = 50,
    ymin = 1e-16,
    ymax = 1e7,
    xminorticks = true,
    yminorticks = true,
    xlabel = {iteration count $j$},
    ylabel style   = {yshift = -.3em},
    scaled x ticks = false,
    x tick label style = {/pgf/number format/1000 sep={\,}},
    y tick label style = {/pgf/number format/1000 sep={\,}},
    cycle list name    = convplotlist,
  ]

  \pgfplotsset{cycle list shift = 1}

\pgfplotstableread{graphics/data/AdvecDiff3000_r30_convTails_stdInit.dat}\tableINPUTdiagTails
  
    \foreach \y in {1}{
      \addplot+ table[x index = 0, y index = \y] {\tableINPUTdiagTails};
    }


\pgfplotstableread{graphics/data/AdvecDiff3000_r30_convErrors_stdInit.dat}\tableINPUTdiagErrors
  
    \foreach \y in {1}{
      \addplot+ table[x index = 0, y index = \y] {\tableINPUTdiagErrors};
    }

\pgfplotstableread{graphics/data/AdvecDiff3000_r30_convTails_truncInit.dat}\tableINPUTtruncTails
  
    \foreach \y in {1}{
      \addplot+ table[x index = 0, y index = \y] {\tableINPUTtruncTails};
    }

\pgfplotstableread{graphics/data/AdvecDiff3000_r30_convErrors_truncInit.dat}\tableINPUTtruncErrors
  
    \foreach \y in {1}{
      \addplot+ table[x index = 0, y index = \y] {\tableINPUTtruncErrors};
    }

\pgfplotstableread{graphics/data/AdvecDiff3000_r30_convTails_eigsInit.dat}\tableINPUTeigsTails
  
    \foreach \y in {1}{
      \addplot+ table[x index = 0, y index = \y] {\tableINPUTeigsTails};
    }

\pgfplotstableread{graphics/data/AdvecDiff3000_r30_convErrors_eigsInit.dat}\tableINPUTeigsErrors
  
    \foreach \y in {1}{
      \addplot+ table[x index = 0, y index = \y] {\tableINPUTeigsErrors};
    }
    
\addplot+ [domain=0:50, samples=2, black] {10e-14};
\node[anchor=south] at (axis cs:8, .5*10e-14) {$\epsilon = 10^{-14}$};

\end{semilogyaxis}
\end{tikzpicture}
\hfill%
\begin{center} 
     \begin{tikzpicture}
  \begin{axis}[%
    hide axis,
    width  = 1mm,
    height = 1mm,
    scale only axis,
    xmin = 0,
    xmax = 1,
    ymin = 0,
    ymax = 1,
    legend columns = 2, 
    legend style   = {
      at     = {(0,0)},
      anchor = center,
      /tikz/every even column/.append style = {column sep = 0.2cm}},
    legend cell align  = {left},
    clip mode          = individual,
    cycle list name    = convplotlist]

    \pgfplotsset{cycle list shift = 1}
    \foreach \y in {1, 2, ..., 6}{
      \addplot+ coordinates{ (0, 0) };
    }
    \addlegendentry{\eqref{eq:H2errConv} for \LQOTSIAdiag{}}
    \addlegendentry{\eqref{eq:H2tailConv} for \LQOTSIAdiag{}}
    \addlegendentry{\eqref{eq:H2errConv} for \LQOTSIAtrunc{}}
    \addlegendentry{\eqref{eq:H2tailConv} for \LQOTSIAtrunc{}}
    \addlegendentry{\eqref{eq:H2errConv} for \LQOTSIAeigs{}}
    \addlegendentry{\eqref{eq:H2tailConv} for \LQOTSIAeigs{}}
  \end{axis}
\end{tikzpicture}
\end{center}
\caption{Convergence behavior of \LQOTSIAdiag{}, \LQOTSIAtrunc{}, and \LQOTSIAeigs{} for $r=30$. The criteria \eqref{eq:H2errConv} and \eqref{eq:H2tailConv} are plotted throughout the iterations for which they are recorded with a tolerance of $\epsilon=10^{-14}$.}
\label{fig:advecdiff_conv}
\end{figure}

As a final point of comparison, we compute a hierarchy of reduced models using \LQOTSIAdiag{}, \LQOTSIAtrunc{}, \\ \LQOTSIAeigs{}, and \LQOBT{} for orders $r=2, 4, \ldots, 30$, and compute the squared relative $\CH_2$ errors~\eqref{eq:relH2error} of the \LQO{}-\ROM{}s. These errors are plotted in~\Cref{fig:h2errors}.
For the \LQO-\TSIA{} \ROM{}s,~\eqref{eq:H2tailConv} is used to determine convergence.
We observe that the error decreases steadily for each method; the errors begin to level off and stagnate for orders not much greater than $r=30$. For smaller orders, the \LQO-\TSIA{} reduced models exhibit a slightly smaller error, although all of the methods produce very small errors for each order of reduction. Despite the different initialization strategies, for a fixed order, the \LQO-\TSIA{} \ROM{}s incur roughly the same error, suggesting that in each case the iteration converges to the same local minimum.


\begin{figure}[t]
\centering
\raggedleft
\begin{tikzpicture}[font = \plotfontsize]
  \pgfplotstableread{graphics/data/AdvecDiff3000_H2errors.dat}\tableINPUT
  
  \begin{semilogyaxis}[%
    width  = .8\linewidth,
    height = .15\textheight,
    scale only axis,
    xmin = 2,
    xmax = 30,
    ymin = 1e-13,
    ymax = 1,
    xminorticks = true,
    yminorticks = true,
    xlabel = {order $r$},
    ylabel = {$\relerr_{\CH_2}$},
    ylabel style   = {yshift = -.3em},
    scaled x ticks = false,
    x tick label style = {/pgf/number format/1000 sep={\,}},
    y tick label style = {/pgf/number format/1000 sep={\,}},
    cycle list name    = errorplotlist
  ]

  \pgfplotsset{cycle list shift = 1}
  
    \foreach \y in {1, 2, 3, 4}{
      \addplot+ table[x index = 0, y index = \y] {\tableINPUT};
    }
  \end{semilogyaxis}
\end{tikzpicture}
\begin{center}
    \begin{tikzpicture}
  \begin{axis}[%
    hide axis,
    width  = 1mm,
    height = 1mm,
    scale only axis,
    xmin = 0,
    xmax = 1,
    ymin = 0,
    ymax = 1,
    legend columns = 2, 
    legend style   = {
      at     = {(0,0)},
      anchor = center,
      /tikz/every even column/.append style = {column sep = 0.2cm}},
    legend cell align  = {left},
    clip mode          = individual,
    cycle list name    = errorplotlist]

    \pgfplotsset{cycle list shift = 1}
    \foreach \y in {1, 2, ..., 4}{
      \addplot+ coordinates{ (0, 0) };
    }
    \addlegendentry{\LQOTSIAdiag{}}
    \addlegendentry{\LQOTSIAtrunc{}}
    \addlegendentry{\LQOTSIAeigs{}}
    \addlegendentry{\LQOBT{}}
  \end{axis}
\end{tikzpicture}
\end{center}
\caption{Relative squared $\CH_2$ errors~\eqref{eq:relH2error} due to the \LQOTSIAdiag{}, \LQOTSIAtrunc{}, \LQOTSIAeigs{}, and \LQOBT{} \ROM{}s for orders $r = 2, 4, \ldots, 30$. For each method, the error decreases steadily as the order $r$ of reduction increases.}
\label{fig:h2errors}
\end{figure}

  
\section{Conclusions}%
\label{sec:conclusions}

In this work, we have presented a novel $\CH_2$-optimality framework for the model reduction of \LQO{} dynamical systems. 
The novel contributions are the computation of gradients of the squared $\CH_2$ error for such systems presented in~\Cref{thm:gradients}, and the generalization the well-known Gramian-based $\CH_2$-optimality conditions for linear dynamical systems to the model reduction of \LQO{} systems, presented in~\Cref{thm:foncs_gramians}. 
Finally, a linear quadratic-output two-sided iteration algorithm, \LQO-\TSIA{}, is proposed in~\Cref{alg:lqo_tsia} for the efficient $\CH_2$-optimal model reduction of \LQO{} systems.
We illustrate the effectiveness of \LQO-\TSIA{} on a numerical example resulting from an optimal control problem. In the future, we will consider alternative $\CH_2$-optimality frameworks for \LQO{} system using the concept of rational function interpolation.


\section*{Acknowledgments}%
\addcontentsline{toc}{section}{Acknowledgments}

The work of Gugercin and Reiter is based upon work supported by the National Science
Foundation under Grant No. AMPS-1923221.
We thank Alejandro Diaz and Matthias Heinkenschloss for providing the code for creating the test examples in~\Cref{sec:numerics}, and are grateful for the feedback of several anonymous referees which that helped to improve the presentation of this manuscript.


\appendix
\section*{Proof of~\Cref{thm:gradients}}
Here, we present the proof of~\Cref{thm:gradients}.
To simplify its presentation, we recall the following Lemma that we will invoke repeatedly.
\begin{lemma}[{\cite[Lemma~A.1]{YanJ99}}]
    \label{lemma:trace_form}
    Let $\BA\in \Rnn$, $\BAr\in\Rrr$ and $\BD,\BF\in\Rnr$ be such that the matrices $\BY,\BW\in\Rnr$ solve the Sylvester equations
    \begin{align*}
        \BA\BY + \BY\BAr^{\trans} + \BD&=\Bzero \\ 
        \mbox{and}~~\BA^{\trans}\BW + \BW\BAr + \BF&=\Bzero,
    \end{align*}
    then, $\trace\left(\BD^{\trans}\BW\right)=\trace\left(\BF^{\trans}\BY\right)$.
\end{lemma}
From~\Cref{lemma:trace_form}, we also have that $\trace\left(\BW^{\trans}\BD\right)=\trace\left(\BY^{\trans}\BF\right)$ by properties of the trace.
In the subsequent result, we take for granted any such identities that arise from cyclic permutations or transposes applied to the result of~\Cref{lemma:trace_form}.
When it is helpful for simplifying notation, we drop the dependence of $\CJ$ on $\Sysred$.

Note that $\Sys$ is fixed. 
We begin by calculating the gradients with respect to the quadratic-output matrices~\eqref{eq:gradM}.
For any $k=1,\ldots,p$, consider an arbitrary infinitesimal $\pertMk\in\Rrr$.
By~\eqref{eq:deriv_as_grad}, it suffices to show the claimed form of the gradient $\gradMk$ in~\eqref{eq:gradM} satisfies
\[\CJ(\Sysred+\pertSys)=\CJ(\Sysred)+\langle\gradMk, \pertMk\rangle_{\frob}+O(\|\pertMk\|_{\frob}^2)\]
where $\pertSys$ denotes the perturbation in the reduced model $\Sysred$ due to $\pertMk$.
The first-order perturbation in $\BMkr$ in turn induces perturbations $\pertZ\in\Rnr$ and $\pertQ\in\Rrr$ in the solutions to~\eqref{eq:constraint_sylv_Z} and~\eqref{eq:constraint_lyap_Qr}.
Expanding upon the perturbed equations~\eqref{eq:constraint_sylv_Z} and~\eqref{eq:constraint_lyap_Qr} reveals that the perturbations $\pertZ$ and $\pertQ$ themselves satisfy
\begin{subequations}
\label{eq:pertM_pertZ}
\begin{align}
    \label{eq:pertM_pertZ_sylv}
    &\BA^{\trans}\pertZ + \pertZ\BAr - \BM_k\BX\pertMk^{\trans}=\Bzero\\
    \begin{split}
    \label{eq:pertM_pertQ_sylv}
    \mbox{and}~~&\BAr^{\trans}\pertQ + \pertQ\BAr + \BMkr\BPr\pertMk^{\trans}\\
    &~~~~~+ \pertMk\BPr\BMkr + 
    \pertMk\BPr\pertMk^{\trans}=\Bzero.
    \end{split}
\end{align} 
\end{subequations}
Note that the term $\pertMk\BPr\pertMk^{\trans}$ is $O(\|\pertMk\|_{\frob}^2)$ in the sense of~\eqref{eq:deriv_as_grad}.
By~\eqref{eq:CJ_tracechar_Q}, the error $\CJ(\Sysred+\pertSys)$ is expanded as
\begin{align*}
    \CJ(\Sysred+\pertSys)&=\CJ(\Sysred) + 2\trace\left(\BB^{\trans}\pertZ\BBr\right) + \trace\left(\BBr^{\trans}\pertQ\BBr\right).
\end{align*}
By applying~\Cref{lemma:trace_form} to the Sylvester equations~\eqref{eq:constraint_sylv_X} and~\eqref{eq:pertM_pertZ_sylv} and using the properties of the trace, the terms in the first-order expansion of $\CJ(\Sysred+\pertSys)$ can be re-written as
\begin{align*}
    \trace\left(\BB^{\trans}\pertZ\BBr\right)=\trace\left(\BBr\BB^{\trans}\pertZ\right)&=\trace\left(-\BM_k\BX\pertMk^{\trans}\BX^{\trans}\right)\\
    &= \trace\left(-\BX^{\trans}\BM_k\BX\pertMk\right).
\end{align*}
\Cref{lemma:trace_form} can be applied to equations~\eqref{eq:constraint_lyap_Pr} and~\eqref{eq:pertM_pertQ_sylv} in a similar way to show that
\begin{align*}
    \trace\left(\BBr^{\trans}\pertQ\BBr\right)&=2\trace\left(\BPr\BMkr\BPr\pertMk\right) + O(\|\pertMk\|_{\frob}^2).
\end{align*}
Substituting these expressions for $\trace\left(\BB^{\trans}\pertZ\BBr\right)$ and $\trace\left(\BBr^{\trans}\pertQ\BBr\right)$ into the expression for $\CJ(\Sys+\pertSys)$, we get
\begin{align*}
    &\CJ(\Sysred+\pertSys)=\CJ(\Sysred) \\
    &~+ \left\langle2\left(\BPr\BMkr\BPr-\BX^{\trans}\BM_k\BX\right),\pertMk \right\rangle_{\frob}+ O(\|\pertMk\|_{\frob}^2).
\end{align*}
So, the gradients with respect to $\BMkr$ satisfy $\gradMk=2\left(\BPr\BMkr\BPr-\BX^{\trans}\BM_k\BX\right)$ for each $k$ as claimed in~\eqref{eq:gradM}.

Next, we compute the gradient $\gradC$ in~\eqref{eq:gradC}. Consider an arbitrary infinitesimal perturbation $\pertC\in\Rpr$ to $\BCr$; let $\pertSys$ be the perturbation to $\Sysred$ corresponding to $\pertC$.
From~\eqref{eq:CJ_tracechar_P}, the first-order expansion of $\CJ(\Sys+\pertSys)$ can be expressed as
\begin{align*}
    \CJ(\Sysred+\pertSys)
    &= \CJ(\Sysred) - 2\trace\left(\BC\BX\pertC^{\trans}\right) \\
    &~~~~~~~~ + 2\trace\left(\BCr\BPr\pertC^{\trans}\right)+ O(\|\pertC\|_{\frob}^2)\\
    =\CJ(\Sysred) + &\left\langle 2\left(\BCr\BPr-\BC\BX\right),@\pertC\right\rangle_{\frob} + O(\|\pertC\|_{\frob}^2).
\end{align*}
So $\gradC=2\left(\BCr\BPr-\BC\BX\right)$ as claimed in~\eqref{eq:gradC}.

We show the formula for $\gradB$ in~\eqref{eq:gradB}.
Consider an arbitrary infinitesimal perturbation $\pertB\in \Rr$ to $\BBr$. This induces perturbations $\pertX\in\Rnr$ and $\pertP\in\Rnr$ in the solutions of~\eqref{eq:constraint_sylv_X} and~\eqref{eq:constraint_lyap_Pr}.
The resulting perturbations satisfy
\begin{subequations}
    \begin{align}
    \label{eq:pertb_pertX_sylv}
        &\BA\pertX + \pertX\BAr^{\trans} + \BB\pertB^{\trans}=\Bzero,\\
        \begin{split}
        \label{eq:pertb_pertP_sylv}
        \mbox{and}~~&\BAr\pertP+ \pertP\BAr^{\trans} + \pertB\BBr^{\trans} \\
        &~~~~~+ \BBr\pertB^{\trans} + O(\|\pertB \|_{\frob}^2)=\Bzero.
        \end{split}
    \end{align}
\end{subequations}
The solutions to~\eqref{eq:constraint_lyap_Qr} and~\eqref{eq:constraint_sylv_Z} depend linearly upon $\BX$ and $\BPr$, so the perturbations $\pertX$ and $\pertP$ induce further perturbations $\pertZ\in\Rnr$ and $\pertQ\in\Rnr$ in the solutions to~\eqref{eq:constraint_sylv_Z} and~\eqref{eq:constraint_lyap_Qr}. These perturbations satisfy
    \begin{subequations}
    \begin{align}
        \label{eq:pertb_pertZ_sylv}
        &\BA^{\trans}\pertZ + \pertZ\BAr -\sum_{k=1}^p\BM_k\pertX\BMkr=\Bzero,\\
        \label{eq:pertb_pertQ_sylv}
        \mbox{and}~~&\BAr^{\trans}\pertQ+ \pertQ\BAr +\sum_{k=1}^p\BMkr\pertP\BMkr=\Bzero.
    \end{align}
\end{subequations}
From~\eqref{eq:CJ_tracechar_Q} we may expand $\CJ(\Sysred+\pertSys)$ as
\begin{align*}
    \CJ(\Sysred+\pertSys)&=\CJ(\Sysred) + 2\trace\left(\BB^{\trans}\BZ\pertB + \BBr\BB^{\trans} \pertZ\right)\\
    + \trace\left(2\BBr^{\trans}\right.&\left.\BQr\pertB + \BBr\BBr^{\trans}\pertQ\right) + O(\|\pertB\|_{\frob}^2).
\end{align*}
We handle the terms in this expansion individually. First, $\trace\left(\BB^{\trans}\BZ\pertB + \BBr\BB^{\trans} \pertZ\right)$ splits into the individual terms $\trace\left(\BB^{\trans}\BZ\pertB\right)$ and $\trace\left(\BBr\BB^{\trans} \pertZ\right)$.
Using properties of the trace and applying~\Cref{lemma:trace_form} to equations~\eqref{eq:constraint_sylv_X} and~\eqref{eq:pertb_pertZ_sylv}, 
 we see that $\trace\left(\BBr\BB^{\trans} \pertZ\right)$ can be written
\begin{align*}
     \trace\left(\BBr\BB^{\trans} \pertZ\right) &= \trace\left(-\left(\sum_{k=1}^p\BMkr\pertX^{\trans}\BM_k\right)\BX\right)\\
     &=\trace\left(-\pertX\sum_{k=1}^p\BMkr\BX^{\trans}\BM_k\right)\\
     &=\trace\left(-\left(\BZ^{\trans}\BA+ \BZ^{\trans}\BAr^{\trans}\right)\pertX \right)+ \trace\left(\pertX\BCr^{\trans}\BC\right)\\
     &=\trace\left(-\BZ^{\trans}\left(\BA\pertX + \pertX\BAr^{\trans}\right)\right) + \trace\left(\pertX\BCr^{\trans}\BC\right)\\
     &=\trace\left(\BZ^{\trans}\BB\pertB^{\trans}\right)+ \trace\left(\pertX\BCr^{\trans}\BC\right)\\
     &=\trace\left(\BB^{\trans}\BZ\pertB\right)+ \trace\left(\pertX\BCr^{\trans}\BC\right).
\end{align*}
Applying~\Cref{lemma:trace_form} to equations~\eqref{eq:constraint_LTI_sylv_Z} and~\eqref{eq:pertb_pertX_sylv}, it follows that
\begin{align*}
    \trace\left(\pertX\BCr\BC^{\trans}\right)=\trace\left(\BCr\BC^{\trans}\pertX\right)&=\trace\left(-\pertB\BB^{\trans}\BZ_1\right)\\
    &=\trace\left(-\BB^{\trans}\BZ_1\pertB\right).
\end{align*}
So the term $2\trace\left(\BB^{\trans}\BZ\pertB + \BBr\BB^{\trans} \pertZ\right)$ in the expression of $\CJ(\Sysred+\pertSys)$ becomes
\begin{align*}
    2\trace\left(\BB^{\trans}\BZ\pertB + \BBr\BB^{\trans} \pertZ\right) &= 4\trace\left(\BB^{\trans}\BZ\pertB\right) -\\
    &~~~~~~~~~~2\trace\left(\BB^{\trans}\BZ_1\pertB\right).
\end{align*}
The term $\trace\left(2\BBr^{\trans}\BQr\pertB + \BBr\BBr^{\trans}\pertQ\right)$ in $\CJ(\Sys+\pertSys)$ can be dealt with using similar ideas to be written as
\begin{align*}
        \trace\left(2\BBr^{\trans}\BQr\pertB + \BBr\BBr^{\trans}\pertQ\right)&=4\trace\left(\BBr^{\trans}\BQr\pertB\right)\\
    -2\trace\left(2\BBr^{\trans}\right.&\left.\BQoner\pertB\right)+O(\|\pertB\|_{\frob}^2).
\end{align*}
The expression for $\CJ(\Sysred+\pertSys)$ in this case becomes
\begin{align*}
    \CJ(\Sysred+\pertSys) &= \CJ(\Sysred) + 4\trace\left(\BB^{\trans}\BZ\pertB\right) - 2\trace\left(\BB^{\trans}\BZ_1\pertB\right)\\
    +4\trace\left(\BBr^{\trans}\right.&\left.\BQr\pertB\right)
    -2\trace\left(\BBr^{\trans}\BQoner\pertB\right) + O(\|\pertB\|_{\frob}^2)\\
    =\CJ(\Sysred) +&\left\langle 2\left(\left(2\BQr-\BQoner\right)\BBr + \left(2\BZ^{\trans}-\BZ_1^{\trans}\right)\BB\right), \pertB\right\rangle_{\frob} \\
    &~~~~~~+ O(\|\pertB\|_{\frob}^2).
\end{align*}
Therefore, $\gradB = 2\left(\left(2\BQr-\BQoner\right)\BBr + \left(2\BZ^{\trans}-\BZ_1^{\trans}\right)\BB\right)$ as claimed in~\eqref{eq:gradB}.

Lastly, we compute $\gradA$ in~\eqref{eq:gradA}.
Consider an arbitrary infinitesimal perturbation $\pertA\in\Rrr$ to $\BAr$.
As was the case for the gradient with respect to $\BBr$, this first-order perturbation induces perturbations $\pertX$, $\pertZ\in\Rnr$, and $\pertP$, $\pertQ\in\Rrr$ in the solutions to~\eqref{eq:constraint_lyap_Pr} --~\eqref{eq:constraint_sylv_Z}. These satisfy
\begin{subequations}
\begin{align}
    \label{eq:pertA_pertX_sylv}
    &\BA\pertX+\pertX\BAr^{\trans}+\BX\pertA^{\trans}+O(\|\pertA\|_{\frob}^2)=\Bzero,\\
    \begin{split}
    \label{eq:pertA_pertP_sylv}
    &\BAr\pertP+\pertP\BAr^{\trans}+\pertA\BPr+\BPr\pertA^{\trans}\\
    &~~~~~+O(\|\pertA\|_{\frob}^2)=\Bzero,
    \end{split}
    \\
    \begin{split}
    \label{eq:pertA_pertZ_sylv}
    &\BA^{\trans}\pertZ+\pertZ\BAr+\BZ\pertA - \sum_{k=1}^p\BM_k\pertX\BMkr\\
    &~~~~~+O(\|\pertA\|_{\frob}^2)=\Bzero,
    \end{split}
    \\
    \begin{split}
    \label{eq:pertA_pertQ_sylv}
    \mbox{and}~~&\BAr^{\trans}\pertQ+\pertQ\BAr+\pertA^{\trans}\BQr+\BQr\pertA\\
    &~~~~~+\sum_{k=1}^p\BMkr\pertP\BMkr+O(\|\pertA\|_{\frob}^2)=\Bzero.
    \end{split}
\end{align}
\end{subequations}
By~\eqref{eq:CJ_tracechar_Q}, the error $\CJ(\Sysred+\pertSys)$ may be expanded as
\begin{align*}
    \CJ(\Sysred+\pertSys)=\CJ(\Sysred) + 2\trace\left(\BBr\BB^{\trans}\pertZ\right) + \trace\left(\BBr\BBr^{\trans}\pertQ\right).
\end{align*}
We deal with the terms in the above expansion individually as follows.
Applying~\Cref{lemma:trace_form} to equations~\eqref{eq:constraint_sylv_X} and~\eqref{eq:pertA_pertZ_sylv}, $\trace\left(\BBr\BB^{\trans} \pertZ\right)$ can be re-written as
\begin{align*}
     \trace\left(\BBr\BB^{\trans} \pertZ\right)&= \trace\left(-\left(\sum_{k=1}^p\BMkr\pertX^{\trans}\BM_k - \pertA^{\trans}\BZ^{\trans}\right)\BX\right)\\
     &~~~~~+O(\|\pertA\|_{\frob}^2)\\
     &= \trace\left(\BX^{\trans}\BZ\pertA\right) - \trace\left(\sum_{k=1}^p\BMkr\BX^{\trans}\BM_k\pertX\right)\\
     &~~~~~+O(\|\pertA\|_{\frob}^2).
\end{align*}
From~\eqref{eq:constraint_sylv_Z} and~\eqref{eq:pertA_pertX_sylv}, observe that
\begin{align*}
    \trace&\left(\sum_{k=1}^p\BMkr\BX^{\trans}\BM_k\pertX\right)=\trace\left(\left(\BZ^{\trans}\BA+\BAr^{\trans}\BZ^{\trans}\right)\pertX\right)\\
    &~~~~~\hspace{1.6in}-\trace\left(\BCr^{\trans}\BC\pertX\right)\\
    &=\trace\left(\BZ^{\trans}\left(\BA\pertX+\pertX\BAr^{\trans}\right)\right) -\trace\left(\BCr^{\trans}\BC\pertX\right)\\
    &=-\trace\left(\BZ^{\trans}\BX\pertA^{\trans}\right)-\trace\left(\BCr^{\trans}\BC\pertX\right)+O(\|\pertA\|_{\frob}^2).
\end{align*}
Applying~\Cref{lemma:trace_form} to~\eqref{eq:constraint_LTI_sylv_Z} and~\eqref{eq:pertA_pertX_sylv} allows us to simplify the $\trace\left(\BCr^{\trans}\BC\pertX\right)$ term further as
\begin{equation*}
    -\trace\left(\BCr^{\trans}\BC\pertX\right)=\trace\left(\pertA\BX^{\trans}\BZ_1\right).
\end{equation*}
So, the term appearing in the expansion of $\trace\left(\BBr\BB^{\trans}\pertZ\right)$ ultimately becomes
\begin{align*}
   -\trace\left(\sum_{k=1}^p\BMkr\BX^{\trans}\BM_k\pertX\right)&=\trace\left(\BX^{\trans}\BZ\pertA\right) \\
    -\trace\left(\BX^{\trans}\BZ_1\pertA\right)
    &+O(\|\pertA\|_{\frob}^2).
\end{align*}
In aggregate, the term $\trace\left(\BBr\BB^{\trans} \pertZ\right)$ in the expansion of $\CJ(\Sysred+\pertSys)$ becomes
\begin{align*}
    \trace\left(\BBr\BB^{\trans}\pertZ\right)&= 2\trace\left(\BX^{\trans}\BZ\pertA\right)  - \trace\left(\BX^{\trans}\BZ_1\pertA\right)\\
    &~~~~~+O(\|\pertA\|_{\frob}^2).
\end{align*}
Using a similar line of reasoning, the term $\trace\left(\BBr\BBr^{\trans}\pertQ\right)$ can be re-written as
\begin{align*}
        \trace\left(\BBr\BBr^{\trans}\pertQ\right)
        &=4\trace\left(\BPr\BQr\pertA\right)-2\trace\left(\BPr\BQoner\pertA\right)\\
        &~~~~~+ O(\|\pertA\|_{\frob}^2).
\end{align*}
Thus, the expansion $\CJ(\Sysred+\pertSys)$ simplifies to
\begin{align*}
    \CJ(\Sysred+\pertSys)&=\CJ(\Sysred)+ 4\trace\left(\BX^{\trans}\BZ\pertA\right) -2\trace\left(\BX^{\trans}\BZ_1\pertA\right)\\
    +4\trace\left(\BPr\BQr\right.&\left.\pertA\right)-2\trace\left(\BPr\BQoner\pertA\right)+O(\|\pertA\|_{\frob}^2) \\
    =\CJ(\Sysred) +&\left\langle 2\left(\left(2\BZ^{\trans}-\BZ_1^{\trans}\right)\BX + \left(2\BQr-\BQoner\right)\BPr\right) , \pertA\right\rangle_{\frob} \\
    &+ O(\|\pertA\|_{\frob}^2).
\end{align*}
Thus, $\gradA=  2\left(\left(2\BZ^{\trans}-\BZ_1^{\trans}\right)\BX + \left(2\BQr-\BQoner\right)\BPr\right)$. This completes the proof.


\addcontentsline{toc}{section}{References}
\bibliographystyle{spmpsci}
\bibliography{root}
\end{document}